\documentclass[11pt, reqno]{article}   	

\usepackage{hyperref}
\usepackage{adjustbox}
\usepackage{amssymb, latexsym}
\usepackage{amsmath}
\usepackage{amsthm}
\usepackage{amsbsy}
\usepackage[toc,page]{appendix}
\usepackage{braket}
\usepackage{breqn}
\usepackage{authblk}
\usepackage{caption}
\usepackage{color}
\usepackage{comment}
\usepackage{dsfont}
\usepackage{esint}
\usepackage{enumerate}
\usepackage{enumitem}
\usepackage[T1]{fontenc}
\usepackage{graphicx}
\usepackage{authblk}
\usepackage{lipsum}
\usepackage{hyperref}
\usepackage{latexsym}
\usepackage{mathrsfs}
\usepackage{mathtools}
\usepackage{subcaption}
\usepackage{float}
\restylefloat{table}
\usepackage{listings}
\usepackage[dvipsnames]{xcolor} 

\usepackage{times}
\usepackage{url} 
\usepackage{tikz}
\usetikzlibrary{cd}
\usetikzlibrary{positioning}

\usepackage[left=2.5cm,top=2.5cm,right=2.5cm,bottom=2.5cm]{geometry}

\numberwithin{equation}{section}
\newtheorem{thm}{Theorem}[section]
\newtheorem{theorem}[thm]{Theorem}
\newtheorem{lemma}[thm]{Lemma}

\newtheorem{assumption}[thm]{Assumption}

\newtheorem{corollary}[thm]{Corollary}

\newtheorem{example}[thm]{Example}

\newtheorem{definition}[thm]{Definition}

\newtheorem{proposition}[thm]{Proposition}

\newtheorem{remark}[thm]{Remark}
\newtheorem{problem}[thm]{Problem}

\def\de{\delta}

\def\ep{\epsilon}

\def\Ga{\Gamma}

\def\la{\lambda}
\def\La{\Lambda}
\def\om{\omega}
\def\Om{\Omega}
\def\pa{\partial}

\def\vphi{\varphi}

\def\N{\mathbb{N}}

\def\R{\mathbb{R}}

\def\ex{\exists}
\def\fa{\forall}
\def\grad{\nabla}
\def\lang{\langle}

\def\rang{\rangle}

\def\Alg{\mathcal{A}}

\def\argmin{\text{argmin}}

\def\cof{\text{cof}}

\def\des{\text{des}}

\def\dist{\text{dist}}
\def\dive{\text{div}}

\def\liminf{\text{liminf}}
\def\limsup{\text{limsup}}

\def\loc{\text{loc}}

\def\supp{\text{supp}}

\newcommand\beq{\begin{equation}}
	\newcommand{\bburl}[1]{\textcolor{blue}{\url{#1}}}
	\newcommand\eeq{\end{equation}}
\newcommand\bea{\begin{eqnarray}}
	\newcommand\eea{\end{xq}}
\newcommand\bi{\begin{itemize}}
	\newcommand\ei{\end{itemize}}
\newcommand\ben{\begin{enumerate}}
	\newcommand\een{\end{enumerate}}

\def\Rn{\mathbb{R}^n}

\everymath{\displaystyle}

\begin{document}
\title{Localization for nonlocal gradient-based optimal control problems}

 \author{Javier Cueto\thanks{Department of Mathematics, Universidad Aut\'onoma de Madrid, Spain. \textbf{Email:} javier.cueto@uam.es}, \ Joshua M. Siktar\thanks{Department of Mathematics, Texas A\&M University, College Station, TX 77843, USA. \textbf{Email:} jmsiktar@tamu.edu}}

\date{}

	\maketitle
    
	\textbf{Abstract:}
		In this paper we consider optimal control problems in the nonlocal function space framework of \cite{bellido2023non}, where there are two different parameters: a horizon parameter $\de > 0$; and a fractional parameter $s \in (0, 1)$. The constraints are given in the form of minimizing an energy density, and we will focus on two particular cases: the well-posed case where the underlying energy density is convex and is given by the nonlocal $p$-Laplacian; and a more general poly/quasiconvex energy for which minimizers exist but may not be unique. 
        The study is concluded by analyzing the approximation to local problems in two parallel ways, either taking the fractional parameter $s$ to $1$ or the horizon parameter $\de$ to $0$.
	
	
	
	\section{Introduction}\label{Sec: intro}

    In this work, we consider families of truncated nonlocal gradient-based optimal control problems with two parameters: a modeling (horizon) parameter $\de > 0$, and a fractional parameter $s \in (0, 1)$. These optimal control problems take the form
    	\begin{equation*}\label{Eq: generic}
		\begin{cases}
			\min\{\mathcal{F}(u, g) \ | \  (u, g) \in H^{s, p, \de}_0(\Om_{-\de}; \R^n) \times Z_{\text{ad}}\} \\
			u \in \argmin_{v \in H^{s, p, \de}_0(\Om_{-\de}; \R^n)}\mathcal{W}^{s, \de}_g(v),
		\end{cases}
	\end{equation*}
    where the states $u \in H^{s,p,\de}_0$ (the energy space of our nonlocal gradients) are required to minimize a nonlocal energy density $\mathcal{W}^{s, \de}_g$. Here $g$, the control variable, belongs to the class of inputs $Z_{\text{ad}}$.
 All of these will be precisely defined in the next section. We also study the convergence of solutions to these control problems as either parameter approaches its local limit, approximating solutions to partial differential equation (PDE)-based optimal control problems.
	
More precisely, we study optimal control problems described via nonlocal gradients, motivated by the increased popularity of nonlocal models in the last decades combined with the framework provided in \cite{bellido2023non} and \cite{mengesha2023control}. Additionally, we expect to recover the classical case through two distinct ways, either by taking the fractional parameter $s$ to $1$, or the horizon $\delta$ to $0$. To our knowledge this is the first work to consider optimal control problems whose constraints involve minimizing poly or quasiconvex energies, even in the local setting.
 
The pursuit of more general models admitting functions exhibiting singularity phenomena that account for long-range interactions has placed nonlocal models in a more prominent position in modern mathematics, and attracted interest from other research fields; they ultimately replace classical derivatives with integral operators.
Notable examples include peridynamics \cite{silling2000reformulation, silling2007peridynamic}, a nonlocal description of solid mechanics proposed by Stewart Silling, where interactions among particles are considered to occur up to a finite distance $\de>0$, (called the \textbf{horizon}); and fractional calculus, where usual derivatives are generalized to an index of order $s \in (0,1)$.
Nonlocality can be showcased mathematically through a wide variety of different structures; notwithstanding, a natural way of proceeding is by introducing nonlocal gradients. This approach also fits within state-based peridynamics \cite{sarego2016linearized, silling2007peridynamic}, as these operators capture the collective information of a difference quotient around a point, which is one of the main ideas behind the deformation state concept of state-based peridynamics (instead of directly computing the energy for each pairwise interaction), where one recovers the horizon $\delta$ to determine the maximum interaction distance from each point. Their mathematical structure can be represented as the integration of a weighted version of the aforementioned quotient by a unit vector: 
	\begin{equation}\label{eq: general nonlocal gradient}
		\mathscr{G}_{\rho}u(x) \ := \ \int_{\Om}\rho(x - y)\frac{u(x) - u(y)}{|x - y|}  \frac{x - y}{|x - y|}dy.
	\end{equation}
	Here $\rho \in L^1(\Om)$ is a kernel function with a singularity at the origin, influencing the order of differentiability of the operator. The most common example of a nonlocal gradient is the \textbf{Riesz $s$-fractional gradient}, which is given by $\Om := \R^n$ and 
	\begin{equation}\label{Eq: Riesz}
		\rho(\xi) \ := \ \frac{c_{n, s}}{|\xi|^{n - 1 + s}},
	\end{equation} 
	with $c_{n, s}$ a normalizing constant. This choice leads to frameworks with a rich structure, enjoying several properties and extending those of the fractional Laplacian \cite{bellido2023non, shieh2015new, shieh2018new, vsilhavy2020fractional}. However, the use of an unbounded domain is inherently unsuitable for solid mechanics problems, so we instead use the following definition:
	\begin{equation}\label{eq: introGrad}
		D^s_{\de}u(x) \ := \ c_{n, s}\int_{B(x, \de)}\frac{u(x) - u(y)}{|x - y|}  \frac{x - y}{|x - y|}\frac{w_{\de}(x - y)}{|x - y|^{n - 1 + s}}dy,
	\end{equation}
	where 
 $w_{\de}: \Om \rightarrow [0, \infty)$ is, a priori, a cut-off function whose properties will be specified later. Nonetheless, those properties were recently generalized in a recent work by Bellido, Mora-Corral and Sch\"onberger \cite{bellido2025nonlocal}. Actually, most of the results presented in this paper would also hold for the kernels $\rho_{\de}$ introduced therein, since they also provided Poincar\'e inequalities as well as continuous and compact embeddings for kernels bounded by fractional singularities. We were initially interested in the localization when $s$ goes to $1$, and therefore considering \eqref{eq: introGrad}, we later added the localization in $\de$ thanks to the results in \cite{cueto2025gamma}. An alternative to the nonlocal gradient-defined energy is the one defined directly via double integration \cite{bellido2014existence, mengesha2012nonlocal, mengesha2014bond, schytt2024dual, scott2019fractional}.

In the analysis of optimal control problems, the controls (denoted by $g$) are [input] variables that influence the state (output) variables (denoted by $u$). Such an influence is usually given by a differential operator (PDE)-based constraint, but we instead use nonlocal operators. For instance, in \cite{mengesha2023control} the search of a desired deformation state minimizing a cost functional was made among all those verifying a linear energy equation given by a nonlocal formulation, such as linear elastic deformation energies from the bond-based description of peridynamics. The generic form of such a problem is
 \begin{equation*}
     \begin{cases}
         &\min \mathcal{F} (u,g) \\
         &\mathcal{L}_\de u \ = \ g \,\text{ in } \Om,
     \end{cases}
 \end{equation*}
 with the pair $(u,g)$ belonging to an admissible set $X_{\text{ad}} \times Z_{\text{ad}}$. In addition, $\mathcal{L}_\de u$ is a nonlocal linear operator playing the role of the Laplacian.

As a more specific example, this framework allows us to consider the following nonlocal $p$-Laplacian problem
\begin{equation} \label{eq: nonlocal p laplacian optimal control introduction}
    \begin{cases}
			\min\mathcal{F}(u, g)  \\
			\Delta^s_{\de,p} u \ := \ \dive^s_\de (|D^s_\de u|^{p-2} D^s_\de u) \ = \ g  \,\text{ in } \Om_{-\de}
		\end{cases}
\end{equation}
with the pair $(u,g)$ belonging to an admissible set $ H^{s, p, \de}_0(\Om_{-\de}; \R^n) \times Z_{\text{ad}}$, $p\in (1,\infty)$, $\dive^s_\de$ being a nonlocal divergence operator, and $H^{s, p, \de}_0(\Om_{-\de}; \R^n)$ are nonlocal Sobolev spaces linked to the nonlocal gradients $D^s_\de$, whose values in the boundary layer are fixed to be $0$. The latter is usually referred to as a \textbf{nonlocal Dirichlet boundary condition}. Thanks to \cite{kreisbeck2024non}, alternative boundary conditions may also be considered, but exploring the range of possibilities here is not the focus of this work. Here we make a mention of papers that consider optimal control problems in the peridynamics setting: \cite{han2024compactness, mengesha2023control, munoz2022local}. We note that unlike the current work, \cite{han2024compactness} uses an optimal control framework possessing kernels with an infinite horizon, for instance with tempered fractional kernels.

In any case, we can identify the state constrains as the Euler-Lagrange equations of a convex energy functional (see \cite{bellido2023non}),
\begin{equation*}
    \begin{cases}
			\min_{(u,g)\in  H^{s, p, \de}_0(\Om_{-\de}; \R^n) \times Z_{\text{ad}}}\mathcal{F}(u, g)
            \\
			\dive^s_\de (D_z W_g(x,u,D^s_\de u))=D_y W_g(x,u,D^s_\de u)  \,\text{ in } \Om_{-\de},
		\end{cases}
\end{equation*}
where $W_g$ is an energy density depending on the choice of control $g$, and $D_y$, $D_z$ stand for the derivatives with respect to the second and third components, respectively. If $W_g(x,u,D^s_\de u)=\frac{1}{p}|D^s_\de u|^p -g$, we  recover \eqref{eq: nonlocal p laplacian optimal control introduction}. This Euler-Lagrange condition can be relaxed by merely requiring that the deformation states belong to the sets of minimizers of the energy functional,
	\begin{equation*}\label{Eq: generic2}
		\begin{cases}
			\min_{(u,g)\in  H^{s, p, \de}_0(\Om_{-\de}; \R^n) \times Z_{\text{ad}}}\mathcal{F}(u, g) \\ 
			u \in \argmin_{v \in H^{s, p, \de}_0(\Om_{-\de}; \R^n)}\mathcal{W}^{s, \de}_g(v),
		\end{cases}
	\end{equation*}
	where $\mathcal{W}^{s, \de}_g(\cdot)$ is a nonlocal energy depending on the parameters $s$ and $\de$, and a choice of control $g$. The advantage of this approach is that the Fréchet differentiability assumption on the energies becomes unnecessary. In addition, we could also make $\mathcal{F}(\cdot, \cdot)$ explicitly depend on the nonlocal gradient to make fuller use of the characterizations of \cite{cueto2023variational}, but that would detract from the main ideas of this paper, so we stick to these assumptions.
    
    We distinguish two classes of nonlocal energies to be considered: the first is a weighted nonlocal $p$-Laplacian energy, which is strictly convex; the second is a general nonlocal quasiconvex energy satisfying typical $p$-growth bounds. It is well-known that the standard notion of convexity is not amenable to finding minimizers of functionals that take vector-valued arguments. Through the 20th century, mathematicians developed alternative notions of convexity that, while weaker than standard convexity, can still be used to show existence of minimizers for functionals. One of these classical notions is known as polyconvexity which, in the three-dimensional case, assumes the cost functional is jointly convex in the gradient matrix of the argument function, its co-factor matrix, and its determinant. The classical Piola identity is utilized to show the weak lower semi-continuity of such functionals, and in turn existence of minimizers can be proven. Weaker still is the notion of quasiconvexity, and this mathematically encapsulates materials with the property that an affine deformation has a smaller energy than any internally distorted deformation. In this case the classical Morrey's Theorem assures existence of minimizers \cite{acerbi1984semicontinuity, marcellini1985approximation, meyers1965quasi, morrey1952quasi}. See also \cite{ball1976convexity} for motivation of the notions of quasi and polyconvexity in the context of nonlinear elasticity.
	Nonlocal analogues of these results, where we replace the classical gradient with different notions of a nonlocal gradient, were proven in \cite{bellido2021gamma, bellido2023minimizers, cueto2023variational}. Particularly those in \cite{bellido2023minimizers, cueto2023variational} will be recalled later for direct use in the current paper. These include nonlocal Poincar\'e inequalities and embedding results.



	As we will see, when the energy being minimized is non-convex (either quasiconvex or polyconvex), we may obtain existence of solutions to the nonlocal optimal control problem. However, the admissible set of controls and states will not, in general, be convex, so we cannot guarantee uniqueness of minimizers. The lack of uniqueness of minimizers for the energy will also prevent us from proving convergence of solutions for the family of control problems as $s \rightarrow 1^-$ or $\de \rightarrow 0^+$.
	
    To treat the analysis as $s \rightarrow 1^-$, we rely on \cite{bellido2023non, cueto2023variational}; meanwhile, for $\de \rightarrow 0^+$, we rely on results in \cite{arroyo2025functional, cueto2025gamma}. These papers establish the convergence of minimizers for our family of energy densities.\\

	 The paper's remaining sections proceed as follows: Section \ref{Sec: prelim} provides notation and functional analytic preliminaries. Then, Section \ref{Sec: Optimal Control} gives the existence and uniqueness results when $\de$ and $s$ are fixed. After that Section \ref{Sec: asym} separately discusses convergence of minimizers and energies as $\de \rightarrow 0^+$ or as $s \rightarrow 1^-$; namely Subsection \ref{subsec: s->1pLap} 
     addresses the $s \to 1^-$ convergence while Subsection \ref{subsec: de->0pLap} 
    addresses the $\de \to 0^+$ convergence.  Finally, Section \ref{sec: future} details directions for future projects. 
     
	
	
	\section{Preliminaries and notation}\label{Sec: prelim}
	
	\subsection{Basic definitions}\label{Subsec: Notation}
	
	Throughout this paper $\Om \subset \R^n$ refers to an open and bounded domain, and $B(x, r)$ denotes a ball in $\R^n$ of radius $r$ centered at $x \in \R^n$. Notice that in order to keep the integral defining the gradient over a symmetric domain, $B(x,\de)$, we need to consider functions defined on a slightly bigger domain, $\Om_{\de}:=\Om+B(0,\de)$, with $\Om $ playing the role of the nonlocal interior domain, and where the sum designates the Minkowski sum of two sets. In addition, the nonlocal Laplacian we will be working with, as a composition of two nonlocal derivatives, will be defined on a slightly smaller domain $\Om_{-\de} := \{x \in \Om: \dist(x, \pa \Om) > \de\}$, meaning we would be dealing with a nonlocal boundary composed of two collars, $\Om_{\de}\backslash \Om$ and $\Om\backslash \Om_{-\de}$, an exterior and interior one, respectively. We also use $\argmin$ to denote a solution to a prescribed minimization problem, and $\mathcal{L}^n$ denotes the Lebesgue measure on $\R^n$, while $\mathcal{B}^n$ denotes the Borel measure on $\R^n$. Finally, we assume that $\mathcal{L}^n(\pa \Om_{-\de}) = 0$.
	
	Here $C^{\infty}_c(\R^n; \R^n)$ denotes, as per usual, the smooth, compactly supported vector fields taking values in $\R^n$. For $p \in (1, \infty)$, we denote by $p'$ the conjugate exponent of $p$, i.e. the solution to the equation $\frac{1}{p'} + \frac{1}{p} = 1$. Naturally, then, we denote the duality pairing between $L^{p'}(\Om; \R^n)$ and $L^p(\Om; \R^n)$ as $\lang \cdot, \cdot\rang_{L^{p'}, L^p}$. We will use the notation $A \preceq B$ to denote that there is a nonessential constant $c>0$, such that $A\leq cB$. In addition, $\cdot: \R^n \times \R^n \rightarrow \R$ will denote the Euclidean dot product; $\otimes: \R^n \times \R^n \rightarrow \R^{n\times n}$ will denote the outer tensor product; and $:$ will denote the Fröbenius inner product on matrices in $\R^{n\times n}$.
	In this framework we will introduce the operators and functional spaces we are going to work with, as well as gather a few relevant results from previous works that we will use throughout the article. Here we consider a one-point nonlocal gradient whose kernel is given by a smoothly truncated version of the Riesz potential \eqref{Eq: Riesz}. These gradients are indeed a particular case of \eqref{eq: general nonlocal gradient} where
	\begin{equation}\label{eq: TruncatedRieszKernel}
		\rho^s_{\de}(x) \ := \ \frac{n-1+s}{\gamma(1-s) |x|^{n-1+s}}w_\de(x) \quad \text{ with } \quad \gamma(s):=\frac{\pi^{\frac{n}{s}}2^s\Gamma\left(\frac{s}{2}\right)}{\Gamma\left( \frac{n-s}{2}\right)}.
	\end{equation}
	Here $\Ga$ is the Euler's gamma function and $\de > 0$ is the horizon of interaction. Sometimes we may write instead 
 \begin{equation*}
		\rho^s_{\de}(x) \ = \ \frac{c_{n,s}}{ |x|^{n-1+s}}w_\de(x) \quad \text{ with } \quad c_{n,s} \:= \ \frac{n+s-1}{\gamma(1-s)}.
	\end{equation*}
 Now we state the precise assumptions our kernels satisfy.

	
\begin{assumption}[Kernel assumptions]\label{Assump: Kernels}
		The family of truncated kernels $\{w_{\de}\}_{\de > 0}$ satisfies the following assumptions:
		\begin{enumerate}
			\item $w_\delta:\R^n\to[0,\infty)$ is radial; i.e. $w_\de(x)=\bar{w}_\de(|x|)$ for some non-negative $\bar{w}_\de\in C_c^\infty([0,\infty))$, with $\supp(\bar{w}_\de)\subset [0,\delta)$.
			\item There is a constant $0<b_0<1$ such that $\bar{w}|_{[0,b_0\de]}=a_0$, where $a_0 := \max_{r\ge 0}\bar{w}_\de(r)$.
			\item
			$\bar{w}_\de (r_1) \geq \bar{w}_\de (r_2)$ whenever $r_1 \leq r_2$.	
   \end{enumerate}
	\end{assumption}
   Note that under these assumptions the kernel $\rho^s_\de$ is integrable, i.e., $\rho^s_\de \in L^1 (\R^n)$. Now we precisely define our notions of nonlocal gradient and divergence (see \cite[Definition 3.1]{bellido2023non}).

	\begin{definition}[Nonlocal gradient and divergence]\label{Def: NLGradAndDiv}
		Let $0\leq s < 1$, $\de>0$ be fixed.
		\begin{enumerate}
			\item For $u\in C^\infty_c(\R^n; \R^n)$ the \textbf{nonlocal gradient} $D^s_\de u$ is defined as
			\begin{equation*}
				D^s_\de u(x) \ := \ \int_{B(x,\de)} \frac{u(x)-u(y)}{|x-y|} \otimes \frac{x-y}{|x-y|}\rho^s_{\de}(x - y) \, dy.
			\end{equation*}
			\item For $\phi\in C^\infty_c(\Rn,\Rn)$ the \textbf{nonlocal divergence} $\dive^s_\de \phi$ is defined as
			\begin{equation}\label{Eq: divPV}
				\dive^s_\de \phi(x) \ := -\ \text{pv}_x\int_{B(x,\de)} \frac{\phi(x)+\phi(y)}{|x-y|}\frac{x-y}{|x-y|}\rho^s_{\de}(x - y)dy.
			\end{equation}
Here the notation $\text{pv}$ refers to an integral calculated in the principal value sense.
		\end{enumerate}
	\end{definition}
 \begin{remark}
     We could have alternatively defined the nonlocal divergence in the following way, as was done in \cite{bellido2023non, cueto2025gamma}:
     \[
		c_{n,s} \ \text{pv}_x\int_{B(x,\de)} \frac{\phi(x)-\phi(y)}{|x-y|}\frac{x-y}{|x-y|}\frac{w_\de(x-y)}{|x-y|^{n+s-1}} dy.
		\]
  Both options are equivalent given the radial symmetry of the kernel. We chose the one in \eqref{Eq: divPV} as it is in accordance with the structure provided by Mengesha-Spector \cite{mengesha2015localization}; in particular, the nonlocal integration by parts holds even when the kernels are not necessarily symmetric.
 \end{remark}
	Similarly to the classical case, we can define our nonlocal spaces as the closure of smooth functions under the proper norm, where the gradient of a function in these spaces is defined by density/limit of a sequence (see \cite{bellido2023non}). 
	
	\begin{definition}\label{Def: HspdeNorm}
		Let $1\leq p <\infty$. Then $H^{s,p,\de}(\Om; \R^n)$ is defined as the closure of $C^{\infty}_c(\R^n; \R^n)$ under the norm
		\begin{equation*}
			\|u\|_{H^{s,p,\de}(\Om; \R^n)}\ := \ \|u\|_{L^p(\Om_{\de}; \R^n)} + \|D^s_\de u\|_{L^p(\Om; \R^{n\times n})}.
		\end{equation*}
		We may also write the semi-norm
		\begin{equation*}
			[u]_{H^{s, p, \de}(\Om; \R^n)} \ := \ \|D^s_\de u\|_{L^p(\Om; \R^{n\times n})}.
		\end{equation*}
		In addition, we denote the set of functions in $H^{s,p,\de}(\Om; \R^n)$ that vanishes outside of  $\Om_{-\de}$ as $H^{s, p, \de}_0(\Om_{-\de}; \R^n)$. We also denote, for any $g \in H^{s, p, \de}(\Om; \R^n)$, the affine space $H^{s, p, \de}_g(\Om_{-\de}; \R^n) \ := \ g + H^{s, p, \de}_0(\Om_{-\de}; \R^n)$.
	\end{definition}
	
	The nonlocal Sobolev space is defined for $s \in [0, 1)$, $p \in [1, \infty)$, with $\Om \subset \R^n$ being open. We will focus on when $\Om$ is bounded. When $\Om=\Rn$ it holds that $H^{s,p,\de}(\Rn; \R^n)$ coincides with the Bessel fractional spaces $H^{s,p}(\Rn; \R^n)$ while $H^{s,p}_0(\Om; \R^n)$ stands for the functions in $H^{s,p}(\Rn; \R^n)$ that vanish in $\Om^c$ (see \cite{cueto2023variational}). The case $p=\infty$ can also be defined in the distributional sense as in \cite{cueto2023variational}.

	\subsection{Functional analytic framework}\label{subsec: framework}
	
	In this Subsection we gather some known results pertaining to the nonlocal Sobolev spaces that will be used to study the optimal control problems. First, here are a few basic observations from the definitions of our function spaces (see \cite[Propositions 2.5-2.6]{bellido2023minimizers} and \cite[Corollary 1]{cueto2023variational}). In some cases we require zero Dirichlet conditions as it is in those spaces where the properties have been ascertained.
	\begin{itemize}
		\item For $p \in [1, \infty)$, $(H^{s, p, \de}(\Om; \R^n), \|\cdot\|_{H^{s, p, \de}(\Om; \R^n)})$ is a separable Banach space, and is Hilbert when $p = 2$.
		\item If $p \in (1, \infty)$, then $(H^{s, p, \de}(\Om; \R^n), \|\cdot\|_{H^{s, p, \de}(\Om; \R^n)})$ is also reflexive. We will focus on this case for the remainder of the paper.
		\item Whenever $1 \leq q \leq p < \infty$, we have the continuous embedding $H^{s, p, \de}(\Om; \R^n) \subset H^{s, q, \de}(\Om; \R^n)$.
		\item If $0 \leq s_1 \leq s_2 \leq 1$ and $p \in (1, \infty)$, then $H^{s_2, p, \de}_0(\Om; \R^n) \subset H^{s_1, p, \de}_0(\Om; \R^n)$ is a continuous embedding with an embedding constant depending only on $n$, $\de$, and $p$.
		 \item Let $p\in (1,\infty)$. For every $\de>0$, $H^{s,p,\de}_0(\Om; \R^n)=H^{s,p}_0(\Om; \R^n)$. Recall that we are not assuming here that $s$ and $\de$ are linked. 
	\end{itemize}
	
	The following nonlocal Poincaré inequality will be essential for our analysis in multiple instances. Initially obtained in  \cite[Theorem 4]{bellido2023non}, by adding an additional hypothesis as in \cite{cueto2023variational} or \cite{cueto2025gamma}, we can choose the constant to be independent of $s$ or $\delta$, respectively, so as to follow the localization process.
	\begin{proposition}[Nonlocal Poincaré inequality]\label{Prop: NLPoincare}
		Let $s \in [0, 1)$,  $\delta > 0$, $p \in (1, \infty)$, $\Om \subset \R^n$ be open and bounded. Then there exists a constant $C > 0$ such that for all $u \in H^{s, p, \de}_0(\Om_{-\de}; \R^n)$, we have
		\begin{equation*}\label{Eq: NLPoincare}
			\|u\|_{L^p(\Om; \R^n)} \ \leq \ C\|D^s_{\de}u\|_{L^p(\Om; \R^{n\times n})}.
		\end{equation*}
        In addition, either of these may alternatively hold:
        \begin{itemize}
            \item[a)] if $a_0:=w_\de(0)=1$, then the constant is independent of $s$,
        \item[b)] if we have a sequence of kernels constructed by re-escalation, i.e. $\rho_\de^s(x)=\de^{-n}\rho \left(\frac{x}{\de}\right)$, with $\rho(x)=\rho^s_1(x)$ given by \eqref{eq: TruncatedRieszKernel} and $\de=1$ in Assumption \ref{Assump: Kernels}, then the constant can be chosen independently of $\de$.
        \end{itemize}
	\end{proposition}
	Another important tool we will need is a compact embedding result, taken from \cite[Theorem 7.3]{bellido2023non} (see also \cite[Theorem 2.9]{bellido2023minimizers}). We only mention the $p \in (1, \infty)$ case here.
	
	\begin{proposition}[Compact embedding]\label{Prop: CompactEmbedding}
		Let $s \in (0, 1)$, $p \in (1, \infty)$, $\de > 0$, and $g \in H^{s, p, \de}(\Om; \R^n)$. Suppose $\{u_j\}^{\infty}_{j = 1} \subset H^{s, p, \de}_g(\Om_{-\de}; \R^n)$ is a sequence such that $u_j \rightharpoonup u$ weakly in $H^{s, p, \de}(\Om; \R^n)$. Then $u_j \rightarrow u$ strongly in $L^q(\Om; \R^n)$, where $q$ satisfies the following:
		\begin{itemize}
			\item $q \in [1, p^*_s)$ if $sp < n$;
			\item $q \in [1, \infty)$ if $sp = n$;
			\item $q \in [1, \infty]$ if $sp > n$.
		\end{itemize}
		Here $p^*_s := \frac{np}{n - sp}$; in addition, we have that $u \in H^{s, p, \de}_g(\Om_{-\de}; \R^n)$.
	\end{proposition}


\subsection{Convexity notions}\label{subsec: convexity}
 
	Finally, for the sake of the reader, we recall concepts weaker than convexity that also provide the weak lower semi-continuity of energy functionals. This was proven to hold in the nonlocal case in \cite{bellido2023minimizers, cueto2023variational}. For the situations where we consider a general energy density (i.e. not specifically the nonlocal $p$-Laplacian case), we will denote it by $W$, and it will sometimes refer to a polyconvex functional, and sometimes to merely a quasiconvex one. We make these definitions precise now.

	\begin{definition}[Polyconvex energy]\label{Def: pConvexEnergy}
		Let $F\in \R^{n\times n}$ and $\phi(F)$ denotes the collection of all its minors. Then, a function $W: \Om \times \R^n \times \R^{n \times n} \rightarrow \R$ is said to be \textbf{polyconvex} if there exists a convex function $\bar{W}$ such that
		\begin{equation*}\label{eq: pConvexEnergy}
			W(x, u, A) \ := \ \bar{W}(x,u,\phi(F)).
		\end{equation*}
	\end{definition}
	In the case of $\R^3$ it reduces to \begin{equation*}
		W(x, u, A) \ := \ \bar{W}(x, u, A, \cof(A), \det(A)),
	\end{equation*}
	where $\bar{W}$ is jointly convex in $A$, $\cof(A)$, and $\det(A)$. A typical example in this case is the following one, discussed in \cite{bellido2020bond, ciarlet1989introduction, davet1986justification, Rin}.
	\begin{example}[Mooney-Rivlin materials]\label{Ex: MooneYRivlin}
		One of the most common examples of a polyconvex function is the family of \textbf{Compressible Mooney-Rivlin materials}, which are functions of the form
		\begin{equation*}\label{Eq: MooneyRivlin}
			f(A) \ := \ a|A|^2 + b|\cof(A)|^2 + T(\det(A))
		\end{equation*}
		where $a, b > 0$ and $T: \R \rightarrow \R \cup \{+\infty\}$ is convex. 
	\end{example}

	
	Now we define quasiconvexity: we use the equivalent notion obtained by replacing the integral over $(0, 1)^n$ with a mean value integral over $\Om$ \cite[Lemma 5.2]{Rin}.  
	
	\begin{definition}[Quasiconvex energy]\label{Def: qConvexEnergy}
		An energy $W: \Om \times \R^n \times \R^{n \times n} \rightarrow \R$ is said to be \textbf{quasiconvex} if for a.e. $x \in \Om$, all $u \in \R^n$, and $A \in \R^{n \times n}$ the following inequality holds for all $\vphi \in W^{1, \infty}_0(\Om)$:
		\begin{equation*}\label{eq: qConvexEnergy}
			W(x, u, A) \ \leq \ \frac{1}{\left|{\Om}\right|}\int_{{\Om}}W(x, u, A + \grad \vphi(y))dy.
		\end{equation*}
	\end{definition}
	It is well-known that a convex energy is automatically polyconvex, and a polyconvex energy in turn is automatically quasiconvex; in addition, these notions all coincide if $n = 1$ \cite{dacorogna2007direct, Rin}. 


	\subsection{Energies}\label{subsec: minEnergy}
   Given that our state constrains are expressed as minimization of energy functionals, 
%
   %
   we present the corresponding notation and existence of minimizer results. 
   For a fixed $\de > 0$, $s\in (0,1)$, and $g \in L^{p'}(\Om; \R^n)$ we define the nonlocal energy $\mathcal{W}^{s, \de}_g: L^p(\Om; \R^n) \rightarrow \R$ via the formula
	\begin{equation*}\label{Eq: NLEnerg}
		\mathcal{W}^{\de, s}_g(u) \ := \ \begin{cases}
			\int_{\Om}W(x, u(x), D^s_{\de}u(x))dx - \lang g, u\rang, \ u \in H^{s, p, \de}(\Om; \R^n) \\
			+\infty \ \text{otherwise},
		\end{cases}
	\end{equation*}
 where we will assume  $W : \Om \times \R^n \times \R^{n \times n} \to \R \cup \{ \infty \}$ satisfies the following:
 \begin{itemize}
     \item $W$ is $\mathcal{L}^n \times \mathcal{B} \times \mathcal{B}^n$-measurable.

\item $W (x, \cdot, \cdot)$ is lower semi-continuous for a.e.\ $x \in \Om$.
 \end{itemize}
 	The key existence of minimizers result needed for our energies has already been proven in the literature, so we adapt it to the energies considered here. 
 Namely, in \cite[Theorem 8.1]{bellido2023non} the existence of minimizers of convex energy functionals was proven under the following lower bound condition on the energy density $W$:
 \begin{equation} \label{eq: convex lower bound}
     W(x,u,A) \ \geq \ a(x) +c|A|^p
 \end{equation}
 for $c>0$, $p\in (1,\infty)$, a.e. $x \in \Om$ and some $a \in L^1(\Om)$. It can easily be seen in that setting that minimizers are unique if $W$ is strictly convex in the third argument. We note that the nonlocal $p$-Laplacian energy studied in Subsection \ref{Sec: pLapEner}
    also falls under the guise of this result. The next proposition verifies that such a result is still true when $gu$ is subtracted from the energy density.

  \begin{proposition}[Minimization of nonlocal energy]\label{minimizationOfNLEnergy}
  Let $W$ be a convex energy density that satisfies the lower bound condition \eqref{eq: convex lower bound}. Then, for any fixed $g \in L^{p'}(\Om; \R^n)$, there exists a minimizer $u \in H^{s, p, \de}_0(\Om_{-\de}; \R^n)$ of $\mathcal{W}^{\de, s}_{g}(\cdot)$. In addition, if $W$ is strictly convex, such a minimizer is unique. 
	\end{proposition}
 \begin{proof}
     In order to apply \cite[Theorem 8.1]{bellido2023non} we just need to see that the hypothesis of the lower bound inequality appearing therein is verified.
     Let $\bar{W}$ be defined as
     \begin{equation*}
         \widehat{W}(x,u,A) \ := \  W(x,u,A) -g u.
     \end{equation*}
     Then, by the lower bound on $W$ together with $\ep$-Young inequality, consecutively, we have that for $\ep>0$,
     \begin{eqnarray*}
 \begin{aligned}
           \widehat{W}(x,u,A) &\geq a(x) +c|A|^p- \frac{|g|^{p'}}{\ep^{p'} p'} -\frac{\ep^p|u|^p}{p} \geq  a(x) +C|u|^p- \frac{|g|^{p'}}{\ep^{p'} p'} -\frac{\ep^p|u|^p}{p} \\
          &= a(x) - \frac{|g|^{p'}}{\ep^{p'} p'} + \left(C-\frac{\ep^p}{p}\right)|u|^p ,
     \end{aligned}
     \end{eqnarray*}
     where we have used the nonlocal Poincar\'e inequality from Proposition \ref{Prop: NLPoincare}. Since $ a(x) - \frac{|g|^{p'}}{\ep^{p'} p'}$ is an integrable function, for $\ep>0$ small enough we have that $C-\frac{\ep^p}{p}>0$ and therefore, the lower bound condition is verified. The proof of uniqueness when $W$ is strictly convex is elementary.
 \end{proof}

 We also comment that, if $W$ is polyconvex or quasiconvex together with respective growth conditions 
 then a similar existence result can be obtained (see \cite[Theorem 6.1]{bellido2023minimizers} or \cite[Corollary 2]{cueto2023variational}, respectively). 

\section{Optimal control}\label{Sec: Optimal Control}

In this section we present the core results: the study of optimal control problems based on nonlocal gradients. For this section we take $s \in (0, 1)$ and $\de > 0$ to be fixed except where otherwise noted. As mentioned in the introduction, our aim is to study the existence of minimizers for optimal control problems of the form
\begin{equation*}
    \begin{cases}
			\min_{(u, g) \in H^{s, p, \de}_0(\Om_{-\de}; \R^n) \times Z_{\text{ad}} } \mathcal{F}(u, g) \qquad \text{ such that }  \\
			\dive^s_\de (D_z W_g(x,u,D^s_\de u))=D_y W_g(x,u,D^s_\de u)  \,\text{ in } \Om_{-\de},
		\end{cases}
\end{equation*}
where $W:\Om \times \R^n \times \R^{n \times n} \rightarrow \R$ is an energy density described in Subsection \ref{subsec: minEnergy} under any of the convexity notions recalled in Subsection \ref{subsec: convexity} and $ Z_{\text{ad}}$, the admissible control space, is defined next.

	\subsection{Assumptions on the cost functional}\label{subsec: AssumpCostfunc}
	Our cost functional $\mathcal{F}: L^p(\Om; \R^n) \times Z_{\text{ad}} \rightarrow \R$ in both the convex and non-convex control problems is of the form
	\begin{equation}\label{costFunc}
		\mathcal{F}(u, g) \ := \ \int_{\Om}F(x, u(x))dx + \int_{\Om} \La(x) |g(x)|^{p'} dx,
	\end{equation}
    where we considered the following assumptions on $F$, $\Lambda$, $g$ and $Z_{\text{ad}}$.

   \begin{assumption}[Assumptions on cost functional]\label{Assump: cost}
   We assume \eqref{costFunc} and its arguments satisfy the following conditions:
    \begin{enumerate}
        \item $g\in Z_{\text{ad}}$, where the admissible control space $Z_{\text{ad}}$ is a nonempty, closed, convex, and bounded subset of $L^{p'}(\Om; \R^n)$ that takes the form
	\begin{equation*}\label{Eq: linearAdSet}
        Z_{\text{ad}} \ := \ \{z \in L^{p'}(\Om; \R^n)  \ | \ \mathfrak{a}(x) \preceq z(x) \preceq \mathfrak{b}(x) \ \text{a.e.} \ x \in \Om\},
	\end{equation*}
	for some vector-valued functions $\mathfrak{a}, \mathfrak{b} \in L^{\infty}(\Om; \R^n)$ \footnote{This $Z_{\text{ad}}$ is known in the literature as a \textbf{box constraint}.}.  
     \item $\La \in L^{\infty}(\Om)$ is a strictly positive function, i.e., there exists $\la > 0$ such that $\La(x) \geq \la$ for all $x \in \Om$.
    \item $F: \Om \times \R^n \rightarrow \R$ is a function satisfying the following assumptions:
	\begin{enumerate}[label=\textbf{F\arabic*}]
		\item\label{F1} For all $v \in \R^n$ the mapping $x \mapsto F(x,v)$ is measurable;
		\item\label{F2} For all $x \in \Om$ the mapping $v \mapsto F(x,v)$ is continuous and convex;
		\item\label{F3} There exist $\ell \in L^1(\Om) $ and $c_1> 0$ such that
		\begin{equation*}\label{Gqcoercive}
			|F(x, v)| \ \leq \ c_1|v|^{p_s^*} + \ell(x)
		\end{equation*}
		for all $x \in \Om$ and all $v \in \R^n$, where $p_s^* := \frac{np}{n-sp}$.
	\end{enumerate}
    \end{enumerate}
    \end{assumption}

	As an example, the assumptions just listed above include the special case
	\begin{equation*}\label{costFuncSpec}
		\mathcal{F}(u, g) \ := \ \frac{1}{p}\|u - u_{\des}\|^p_{L^p(\Om; \R^n)} + \frac{\la}{p'}\|g\|^{p'}_{L^{p'}(\Om; \R^n)}
	\end{equation*}
	where $\la > 0$ is a regularization parameter, and $u_{\des} \in L^p(\Om; \R^n)$ is fixed.

Generically, the optimal control problem we are studying can then be stated as solving
 \begin{equation*}
     \mathcal{F}(\bar{u},\bar{g}) \ = \ \min_{(u,g)\in\mathcal{A}} \mathcal{F}(u,g),
 \end{equation*}
	where $\mathcal{A}$ is the admissible class of pairs, given by
 \begin{equation*}
     \mathcal{A} \ := \ \{ (u,g) \in H^{s,p,\de}_0(\Om_{-\de},\Rn)\times Z_{ad} \, | \, u \in \argmin_{v \in H^{s,p,\de}_0(\Om_{-\de},\Rn)} \mathcal{W}^{\de,s}_g (v) \}.
 \end{equation*}
 Again, $\mathcal{W}^{\de,s}_g$ is an energy functional with the properties described in Subsection \ref{subsec: minEnergy}. Specific use cases of this formulation will be presented throughout the remainder of the document.



		\subsection{Optimal control with convex $p$-Laplacian energy}\label{Sec: pLapEner}
		This section is concerned with optimal control problems with state equation energy corresponding to that of the nonlocal $p$-Laplacian, i.e,. the $p$-Dirichlet nonlocal energy, as its special structure allows us for a more complete analysis (including uniqueness of solutions), than in the non-convex case. Concretely, the inhomogeneous nonlocal $p$-Laplacian corresponds to the energy \eqref{Eq: NLEnerg} whose energy density is given by
	\begin{equation}\label{Eq: pLapDens}
		W(x, u, A) \ := \ \frac{1}{p}\mathbb{A}(x)|A|^{p - 2}A : A,
	\end{equation}
    where $|\cdot|$ denotes the Fröbenius norm of a matrix. Moreover, the coefficient $\mathbb{A} \in L^{\infty}(\R^n; \R^{n\times n})$ is uniformly elliptic and symmetric: $\mathbb{A}(x) = \mathbb{A}(x)^T$ for all $x \in \R^n$, and there exists $\mu > 0$ so that
	\begin{equation}\label{Eq: UnifEllip}
		\mathbb{A}(x)\xi : \xi \ \geq \ \mu|\xi|^2 \ \quad \fa \, x \in \R^n, \ \xi \in \R^{n \times n}.
	\end{equation}
  This implies that the spaces on which our energies are finite are unambiguous with respect to the choice of $\mathbb{A}$.

  For the study of the nonlocal $p$-Laplacian energy, we explore the equilibrium conditions. In particular, we consider the directional derivative, i.e., we will denote, for $s \in (0, 1)$ and $\de > 0$, the mapping $\mathcal{Y}_{\de, s}: H^{s, p, \de}(\Om; \R^n) \times H^{s, p, \de}(\Om; \R^n) \rightarrow \R$ via
	\begin{equation*}\label{Eq: Ydes}
		\mathcal{Y}_{\de, s}(u, v) \ := \ \int_{\Om}\mathbb{A}(x)|D^s_{\de}u(x)|^{p - 2}D^s_{\de}u(x) : D^s_{\de}v(x)dx.
	\end{equation*}
  We start by providing the formal problem statements in the special case where our energy takes a convex, nonlocal $p$-Laplacian form. 
  This is arguably the simplest case in which our analysis is meaningful.
  
  We say that a pair $(u, g) \in H^{s, p, \de}_0(\Om_{-\de}; \R^n) \times Z_{\text{ad}}$ verifies the  nonlocal state (constraint) system in weak form if
  \begin{equation}\label{stateEqnWkFormNL}
		\begin{cases}
			\mathcal{Y}_{\de, s}(u, v) \ = \ \lang g, v\rang \quad \ \fa v \in H^{s, p, \de}_0(\Om_{\de}; \R^n) \\
			u \ = \ 0, \ x \in \Om_{\de}\setminus \Om_{-\de}
		\end{cases}.
	\end{equation}
  Alternatively, its strong form can be written as
  \begin{equation*}\label{stateEqnStrongFormNL}
		\begin{cases}
			-\dive^s_{\de}(\mathbb{A}(x)|D^s_{\de}u(x)|^{p - 2}D^s_{\de}u(x)) \ = \ g(x), \ x \in \Om \\
			u \ = \ 0, \ x \in \Om_{\de}\setminus \Om_{-\de}
		\end{cases}.
	\end{equation*}
	We will call this equation the \textbf{inhomogeneous nonlocal $p$-Laplacian equation}.

	\begin{problem}[Convex nonlocal problem]\label{nlProbStatement}
    Our \underline{convex} nonlocal problem is to solve
		\begin{equation*}\label{stateMinNL}
			\mathcal{F}(\overline{u}, \overline{g}) \ = \ \min_{(u, g) \in \Alg^{\de}_s}\mathcal{F}(u, \ g),
		\end{equation*} 
		where the admissible class of pairs is given by
		\begin{eqnarray*}\label{nonlocalAdmiClass}
			\begin{aligned}
				\Alg^{\de}_s \ := \{(u, g) \in H^{s, p, \de}_0(\Om_{-\de}; \R^n) \times Z_{\text{ad}} \ |  \, (u,g) \ \text{ verifies } \ \eqref{stateEqnWkFormNL} \}.
			\end{aligned}
		\end{eqnarray*}
	\end{problem}	
		%
        %
        The first thing we notice is that the set $\mathcal{A}^\de_s$ is convex, as a consequence of the convexity of the energy.
        \begin{lemma} \label{le: admissible set convex}
            The set $\mathcal{A}^\de_s$ is convex.
        \end{lemma}
        \begin{proof}
            Let us assume we have two admissible pairs, $(u_1,g_1),\, (u_2,g_2) \in \mathcal{A}^{\de}_s$. Then, if we consider a convex combination, $(\tilde{u},\tilde{g})=\la(u_1,g_1) +(1-\la)(u_1,g_1)$, $(\tilde{u},\tilde{g})$, clearly $(\tilde{u},\tilde{g})\in H^{s, p, \de}_0(\Om_{-\de}; \R^n) \times Z_{\text{ad}}$. In addition, by the convexity of $W$ in \eqref{Eq: pLapDens} and linearity of the energy densities with respect to $g$, we have that
    \begin{equation} \label{eq: convex energy inequality}
        \mathcal{W}^{\de,s}_{\tilde{g}}(\tilde{u}) \leq \la \mathcal{W}^{\de,s}_{g_1}(u_1) + (1-\la)\mathcal{W}^{\de,s}_{g_2}(u_2).
    \end{equation}
     If we use now that $u_1$ and $u_2$ are minimizers of $\mathcal{W}^{\de,s}_{g_1}$ and $\mathcal{W}^{\de,s}_{g_2}$, respectively, we have that \eqref{eq: convex energy inequality} leads to
      \begin{equation*}
        \mathcal{W}^{\de,s}_{\tilde{g}}(\tilde{u}) \leq \la \mathcal{W}^{\de,s}_{g_1}(v) + (1-\la)\mathcal{W}^{\de,s}_{g_2}(v)=\mathcal{W}^{\de,s}_{\tilde{g}}(v),
    \end{equation*}
    for every $v \in H^{s,p,\de}_0(\Om_{-\de}; \R^n)$, where the last equality is a consequence of the energy defined by \eqref{Eq: NLEnerg} and the linearity with respect to $g$. Hence, $\tilde{u}$ is a minimzer of $\mathcal{W}^{\de,s}_{\tilde{g}}$, and therefore, $(\tilde{u},\tilde{g})$ is an admissible pair.
        \end{proof}
        The existence of a unique minimizer of $\bar{W}(u)=\frac{1}{p}\mathcal{Y}_{\de, s}(u, u) - \int_{\Om} g \cdot u$ is guaranteed by applying Proposition \ref{minimizationOfNLEnergy} to energy density \eqref{Eq: pLapDens}.  Given that $\bar{W}(x,\cdot,\cdot)\in C^1$ for a.e. $x\in \Om$, such a minimizer is the weak solution to the equation \eqref{stateEqnWkFormNL} (see \cite[Theorem 8.2]{bellido2023non}). This existence and uniqueness allows us to write the state variables in the cost functional in terms of the control ones, which motivates the next definition.
        \begin{definition}
            We define the operator   $S_{\de, s}: L^{p'}(\Om; \R^n) \rightarrow H^{s, p, \de}_0(\Om_{-\de}; \R^n)$ as the solution mapping such that the pair $(g,S_{\de,s}(g))$ verifies \eqref{stateEqnWkFormNL}.
        \end{definition}
     In particular, we have that such a solution mapping is compact.

\begin{lemma}\label{compactnessNLStateOperator}
	The operator $S_{\de,s}: L^{p'}(\Omega; \R^n) \rightarrow H^{s,p,\de}_0(\Omega_{-\de}; \R^n) \hookrightarrow L^p(\Om; \R^n)$ is compact. 
\end{lemma}
\begin{proof}
	For each right-hand side $g \in L^{p'}(\Omega; \R^n)$, let $u$ be the unique solution of \eqref{stateEqnWkFormNL}, $u \in H^{s,p,\de}_0(\Om_{-\de}; \R^n)$. Then, by the uniform ellipticity estimate \eqref{Eq: UnifEllip} and H\"older's inequality ,
	\begin{eqnarray*}
 \begin{aligned}
		\|D^s_\de u\|_{L^p(\Omega; \R^{n \times n})}^p &\lesssim \int_{\Om}\mathbb{A}\left| D^s_\de u \right|^{p - 2}D^s_{\de} u: D^s_{\de} u=\int_{\Omega} g \cdot u \\ &\leq \|g\|_{L^{p'}(\Omega; \R^n)} \| u\|_{L^p(\Omega; \R^n)} \leq C\|g\|_{L^{p'}(\Omega; \R^n)} \| D^s_\de u\|_{L^p(\Omega; \R^{n \times n})},
  \end{aligned}
	\end{eqnarray*}
	where $C>0$ is the constant from the Nonlocal Poincar\'e inequality \eqref{Eq: NLPoincare}.
	Thus, 
	\begin{equation*} 
		\|D^s_\de u\|_{L^p(\Omega; \R^{n\times n})} \leq C \|g\|_{L^{p'}(\Omega; \R^n)}^{\frac{1}{p-1}}.
	\end{equation*}
	Finally, applying Poincar\'e's inequality again yields
	\begin{equation*} \label{eq: S bounded operator}
		\|u\|_{H^{s,p,\de}(\Omega; \R^n)} \leq (C+1) \|D^s_\de u \|_{L^p(\Omega; \R^{n\times n})} \leq C \|g\|_{L^{p'}(\Omega; \R^n)}^{\frac{1}{p-1}},
	\end{equation*}
	implying that $S_{\de, s}$ is bounded. The compactness of $S_{\de,s}$ is a consequence of the composition of a bounded operator with the compact embedding of $H^{s,p,\de}_0(\Om_{-\de}; \R^n)$ into $L^p(\Om; \R^n)$.
\end{proof}
Next, we show a slight generalization of the abstract minimization result \cite[Theorem 4.1]{mengesha2023control}, for use in our nonlocal framework (and later, for the local convex problem as well). Unlike that paper, however, we absorb the appearance of a regularization parameter $\la$ into the weight function $\La$.

\begin{theorem}[Existence]\label{Th: abstractMinimizationConvex}
	Let $S: L^{p'}(\Om; \R^n) \to L^p(\Om; \R^n)$ be a compact operator and $G: L^p(\Om; \R^n)\to \R$ be lower semi-continuous. Let $Z_{ad}$ be defined as in \eqref{Eq: linearAdSet}, and define $j: Z_{ad}\to \R$ by
	\begin{equation*}
		j(g) \ := \ G(Sg) + \int_{\Om}\La(x) |g(x)|^{p'}\,dx
	\end{equation*}
	for some positive $\La \in L^1(\Om)$. Then, the reduced optimization problem
	\begin{equation*}
		\min_{g\in Z_{ad}} j(g)
	\end{equation*}
	has a solution $\tilde{g}$. 
\end{theorem}
\begin{proof}
	Firstly, we prove that $j$ has an infimum. Since the second term is non-negative, we argue by contradiction to show that $G$ is bounded from below. Thus, let us assume there is a sequence $\{w_k\}_{k=1}^{\infty}\subset Z_{ad}$ such that
	\begin{equation*}
		G(Sw_k) \ < \ -k
	\end{equation*}
	for every $k\in \N^+$. As $Z_{ad}$ is a bounded, closed subset of $L^{p'}(\Om; \R^n)$, we have that there is a sub-sequence (not re-labeled) and $\bar{w} \in Z_{ad}$ such that $w_k \rightharpoonup \bar{w}$ in $L^{p'}(\Om; \R^n)$. Since $S$ is a compact operator, we have that $Sw_k \to S\bar{w}$ as $k\to \infty$. By the lower semi-continuity of $G$, we have that
	\begin{equation*}
		G(S\bar{w})\leq \liminf_{k\to \infty} G(Sw_k) \ = \ -\infty,
	\end{equation*}
	reaching a contradiction, as $G$ only assumes finite values.
	
	Secondly, we prove that such an infimum is indeed a minimum. Hence, let $j_0:= \inf_{g\in Z_{ad}} j(g)$, and then there exists a sequence $\{g_k\}_{k=1}^\infty\subset Z_{ad}$ such that $\lim_{k\to \infty} j(g_k)=j_0$ as $k\to \infty$. Again, since $Z_{ad}$ is closed and bounded, there is a sub-sequence (not re-labeled) and $\bar{g}\in Z_{ad}$ such that $g_k \rightharpoonup \bar{g}$ in $L^{p'}(\Om; \R^n)$. We then conclude existence as in \cite[Theorem 4.1]{mengesha2023control}. 
    
	
\end{proof}
We remark that for this result, having $\La \in L^1(\Om)$ is sufficient since $Z_{\text{ad}}$ is a bounded subset of $L^{\infty}(\Om; \R^n)$; however, to complete the program we will later need that $\La \in L^{\infty}(\Om)$. Now we conclude with the well-posedness result for the convex nonlocal control problem.

\begin{theorem}[Well-posedness of nonlocal control problem]\label{WPNLControlThm}
	There exists a pair $(\overline{u}, \overline{g}) \in H^{s, p, \de}_0(\Om_{-\de}; \R^n) \times Z_{\text{ad}}$ solving Problem \ref{nlProbStatement}. In addition, this solution is unique.
\end{theorem}
\begin{proof}
Let the reduced cost functional $j_{\de, s}: L^{p'}(\Om; \R^n) \rightarrow \R$ be defined as 
\begin{equation*}\label{NLReducedCostFunctional}
	j_{\de, s}(g) \ := \ \mathcal{F}(S_{\de,s}(g),g) \ = \ \int_{\Om}F(x, S_{\de, s}(g)(x))dx + \int_{\Om}\La(x)|g(x)|^{p'}dx.
\end{equation*}
	Then, since $S_{\de,s}$ is compact by Lemma \ref{compactnessNLStateOperator}, the first part of the result follows by direct application of Theorem \ref{Th: abstractMinimizationConvex}. 

    As for the uniqueness part, suppose there were two admissible pairs, $(u_1,g_1),\, (u_2,g_2) \in \mathcal{A}^{\de}_s$, that minimize the cost functional \eqref{costFunc}. Then, $(\tilde{u},\tilde{g}):=\frac{1}{2}(u_1,g_1) +\frac{1}{2}(u_2,g_2) \in \mathcal{A}^\de_s$ by the convexity of $\mathcal{A}^{\de}_s$ (Lemma \ref{le: admissible set convex}). Recall that by the strict convexity of the energy $W^{\de,s}_g(\cdot)$, if $g_1=g_2$, then $u_1=u_2$ (there can not be two minimizers of the energy when $g$ is fixed), so we can assume $g_1 \neq g_2$. Nevertheless, by the convexity of $F$ and strict convexity of the weighted $p'$-norm of $g$ in \eqref{costFunc} we would have
    \begin{equation*}
        \mathcal{F}(\tilde{u},\tilde{g}) < \frac{1}{2}\mathcal{F}(u_1,g_1) +\frac{1}{2}\mathcal{F}(u_2,g_2)
    \end{equation*}
    which contradicts $(u_1,g_1)$ and $(u_2,g_2)$ minimizing the cost functional $\mathcal{F}(u,g)$.
    
\end{proof}
 The classical, local analogue of Theorem \ref{WPNLControlThm} is included in Section \ref{Sec: asym}.
\begin{remark}
    All of these results hold for a general (strictly) convex energy density, except for the forthcoming strong convergence of the nonlocal gradients of minimizers to the local one (which is part of Theorems \ref{convMinss->1} and \ref{De->0: convMinss})
\end{remark}



\subsection{Optimal control with non-convex energy}\label{Sec: nonconvex}

Now, we address in this section the optimal control problem where we relax the convexity condition on the energy to being just quasiconvex or policonvex. We do not write an Euler-Lagrange system associated with this constraint, because we gain no benefit by assuming that the energy is (Fréchet, or even Gâteaux) differentiable. Thus, we say that a pair $(u,g)$ verifies the nonlocal state constraint if 
\begin{equation} \label{eq: NL state constrain}
    W^{\de, s}_g(u) \leq W^{\de, s}_g(v) \qquad \forall v\in H^{s,p,\de}_0(\Om_{-\de},\Rn).
\end{equation}
Throughout this section we will assume that $W^{\de,s}_g(\cdot)$ is not identically infinity in $H^{s,p,\de}_0(\Om_{-\de}; \R^n)$ and whose energy density $W:\Om \times \Rn \times \R^{n \times n} \to \R$ verifies one of the following.

		\begin{enumerate}[label=\textbf{H\arabic*}]
			\item\label{H1} There exist $c_1,c_2>0$, $a\in L^1(\Om)$ and Borel function $h:[0,\infty) \to [0,\infty)$ such that $\lim_{t \to \infty} \frac{h(t)}{t}=\infty$, and  $W(x,y,A)\geq a(x) + c_1|F|^p+c_2|\text{cof}(A)|^q + h(|det(A)|)$.
			\item\label{H2} There exist $c, c_0, C_2 > 0$ such that $c|A|^p-c_0 \leq W(x, y, A) \leq C_2(1+|z|^p+|A|^p)$. Moreover, $W$ is assumed to be a Carathéodory integrand.
		\end{enumerate}

	\begin{problem}[Non-convex nonlocal problem]\label{nlProbStatementG}
    Our nonlocal non-convex problem is to find
		\begin{equation*}\label{stateMinNLG}
			\mathcal{F}(\overline{u}, \overline{g}) = \min_{(u, \ g) \in \mathcal{K}^{\de}_s}\mathcal{F}(u, \ g),
		\end{equation*} 
		where the admissible class of pairs $\mathcal{K}^{\de}_s$ is given by
		\begin{equation*}\label{nonlocalAdmiClassG}
			\mathcal{K}^{\de}_s \ := \ \{(u, g) \in H^{s, p, \de}_0(\Om_{-\de}; \R^n) \times Z_{\text{ad}} \ | \ (u,g) \, \text{ verifies}  \, \eqref{eq: NL state constrain}\}.
		\end{equation*}
	\end{problem}

Given the lack of a well-defined solution mapping, unlike the problems studied in Section \ref{Sec: pLapEner}, we can no longer rely on the techniques employed in those cases, including the uniqueness arguments. To prove that Problem \ref{nlProbStatementG} has a solution, we will proceed by the direct method. However, due to the non-uniqueness of minimizers for the energy $\mathcal{W}^{\de, s}_0(\cdot)$, we need an additional tool compared to the convex case. Namely, we need to prove that the collection of admissible states is bounded. In particular, the next two lemmas allow us to verify that a minimizing sequence $\{(u_k,g_k)\}^{\infty}_{k = 1}$ of the cost functional will have a converging sub-sequence in $\mathcal{K}^{\de}_s$. 
\begin{lemma}[Boundedness of admissible states]\label{bddAdmissibleState}
Suppose \eqref{eq: convex lower bound} is satisfied. Then, the image of $\mathcal{K}^{\de}_s$ into $H^{s, p, \de}_0(\Om_{-\de}; \R^n)$, i.e., the collection of states which minimize the energy density for some admissible control given by
	\begin{equation*}\label{eq:UNLAdmis}
		U_{\de, s} \ := \ \{u \in H^{s, p, \de}_0(\Om_{-\de}; \R^n), \ | \ \ex g \in Z_{\text{ad}}, \ (u, g) \in \mathcal{K}^{\de}_s\},
	\end{equation*}
	is bounded. 
\end{lemma}

\begin{proof}
    Let $g \in Z_{\text{ad}}$ be arbitrary and $u$ a minimizer of $\mathcal{W}^{\de, s}_{g}$ in $H^{s, p, \de}_0(\Om_{-\de}; \R^n)$. Then, using the zero function as a test function we have
	\begin{equation}\label{eq: bddAdmissibleStateEq2}
		\mathcal{W}^{ \de,s}_{0}(u)  \ \leq \ W_0^{\de,s}(0). 
	\end{equation}
   If we combine now \eqref{eq: bddAdmissibleStateEq2} with the lower bound \eqref{eq: convex lower bound} on $W$  we get
    \begin{equation*}
     c\|D^s_\de u\| \leq \langle g_{\de, s}, u_{\de, s}\rangle-\int_{\Om} a dx - W_0^{\de,s}(0)
    \end{equation*}
    for some $c>0$ and $a\in L^1(\Om)$. Then, by the lower bound \eqref{eq: convex lower bound} on $W$, Hölder's inequality, the Nonlocal Poincaré Inequality (Proposition \ref{Prop: NLPoincare}), and the boundedness of $Z_{\text{ad}}$, we have
      \begin{eqnarray*}
 \begin{aligned}
            c\|D^s_{\de}u\|^p_{L^p(\Om; \R^{n \times n})} \ &\leq \ \|g_{\de, s}\|_{L^{p'}(\Om; \R^n)}\|u\|_{L^p(\Om; \R^n)} - \int_{\Om} a dx  - W_0^{\de,s}(0)\\ 
            &\leq  \ C\|D^s_{\de}u\|_{L^p(\Om; \R^{n \times n})} + \|a\|_{L^1(\Om)} + \left| W_0^{\de,s}(0)\right|,
     \end{aligned}
     \end{eqnarray*}
       which completes the proof.
\end{proof}
 The next lemma assures that minimization of energies is preserved under the topology on which a minimizing sequence for the cost functional converges, which means the limit of the minimizing sequence will be admissible. A similar result is presented in \cite[Lemma 3.3]{siktar2024superlinear}.
%
\begin{lemma}\label{closureOfMinimization}
    Let $W$ be a weakly lower semi-continuous energy density verifying \eqref{eq: convex lower bound}. Then the set $\mathcal{K}^{\de}_s$ is weakly closed on $H^{s, p, \de}(\Om; \R^n) \times L^{p'}(\Om; \R^n)$ .
\end{lemma}

\begin{proof}
Given $\{(u_k, g_k)\}^{\infty}_{k = 1} \subset \mathcal{K}^{\de}_s$ such that $g_k \rightharpoonup \overline{g}$ weakly in $L^{p'}(\Om; \R^n)$, and $u_k \rightharpoonup \overline{u}$ weakly in $H^{s, p, \de}(\Om; \R^n)$, we want to prove that $(\overline{u}, \overline{g}) \in \mathcal{K}^{\de}_s$. On one hand, since $u_k \in H^{s,p,\de}_0(\Om_{-\de}; \R^n)$, by Proposition \ref{Prop: CompactEmbedding}, we must have that $u_k \rightarrow \overline{u}$ strongly in $L^p(\Om; \R^n)$.
On the other hand, by the definition of $ \mathcal{K}^{\de}_s$, we have that for each $k$,
	\begin{equation}\label{Eq: closureOfMinimizationEq1}
		\mathcal{W}^{\de, s}_{g_k}(u_k) \ \leq \ \mathcal{W}^{\de, s}_{g_k}(v), \quad \forall \, v \in H^{s, p, \de}_0(\Om_{-\de}; \R^n).
	\end{equation}
    Then, by \eqref{eq: convex lower bound} and the weak lower semi-continuity of $W$ we pass to the limit as $k \rightarrow \infty$ in \eqref{Eq: closureOfMinimizationEq1} to obtain
	\begin{equation*}\label{Eq: closureOfMinimizationEq2}
		\mathcal{W}^{\de, s}_{\bar{g}}(\bar{u}) \ \leq \ \mathcal{W}^{\de, s}_{\bar{g}}(v), \quad  \forall \, v \in H^{s, p, \de}_0(\Om_{-\de}; \R^n).
	\end{equation*}
	Consequently, $(\overline{u}, \overline{g}) \in \mathcal{K}^{\de}_s$, completing the proof.
\end{proof}


Recall that weak lower semi-continuity is guaranteed in this nonlocal context assuming quasiconvexity, \cite[Theorem 5]{cueto2023variational}, and polyconvexity, \cite[Theorem 6.1]{bellido2023minimizers}, under the corresponding growth conditions on the energy density.
Now we state and prove the existence result for the non-convex control problem.

\begin{theorem}[Existence of solutions to non-convex control problem]\label{existenceMinsNLGeneral}
    Assume that the energy density $W$ verifies one of the following:
    \begin{itemize}
        \item[a)] $W$ is polyconvex and \ref{H1} is satisfied with $p\geq n-1$ and $q\geq \frac{n}{n-1}$.
        \item[b)] $W$ is quasiconvex and \ref{H2} is satisfied.
    \end{itemize}
Then, Problem \ref{nlProbStatementG} has a minimizer $(\overline{u}, \overline{g}) \in \mathcal{K}^{\de}_s$.
\end{theorem}
\begin{proof}
	We must first show that the cost functional $\mathcal{F}$ defined in \eqref{costFunc} is bounded from below. Since the second term of $\mathcal{F}$ is non-negative, it suffices to show that the first term is bounded from below. Due to the lower bound of inequality \eqref{Gqcoercive}, for any $(u, g) \in \mathcal{K}^{\de}_s$, we have that
	\begin{equation*}\label{eq: existenceMinsNLGeneralEq1}
		\mathcal{F}(u, g) \ \geq \ \int_{\Om}F(x, u(x))dx \ \geq \ -c_1\|u\|^{p^*_s}_{L^{p^*_s}(\Om; \R^n)} - \|\ell\|_{L^1(\Om)}.
	\end{equation*}
	Due to the continuous embedding $H^{s, p, \de}_{g}(\Om_{-\de}; \R^n) \subset L^{p^*_s}(\Om; \R^n)$ of \cite[Theorem 2.8]{bellido2024minimizers}, we actually have that
	\begin{equation*}\label{eq: existenceMinsNLGeneralEq2}
		\mathcal{F}(u, g) \ \geq \ -\|D^s_{\de}u\|^{p^*_s}_{L^p(\Om; \R^{n\times n})} - \|\ell\|_{L^1(\Om)}.
	\end{equation*}
	From this inequality, owing to Lemma \ref{bddAdmissibleState}, we conclude that $\mathcal{F}$ is bounded from below on $\mathcal{K}^{\de}_s$.
	
	Now we prove that the infimum is attained by a converging subsequence. Let $m := \inf_{(u, g) \in \mathcal{K}^{\de}_s}\mathcal{F}(u, g) > -\infty$. Let $\{(u_{ k}, g_{k})\}^{\infty}_{k = 1} \subset \mathcal{K}^{\de}_s$ denote a minimizing sequence for $\mathcal{F}$ such that $\lim_{k \rightarrow \infty}\mathcal{F}(u_{ k}, g_{ k}) = m$. Since $\{g_{k}\}^{\infty}_{k = 1} \subset Z_{ad}$, there exists a $\overline{g} \in Z_{ad}$ such that, up to a not relabeled sub-sequence, $g_{ k} \rightharpoonup \overline{g}$ weakly in $L^{p'}(\Om; \R^n)$. Meanwhile, due to Lemma \ref{bddAdmissibleState} the sequence $\{u_{ k}\}^{\infty}_{k = 1}$ is a bounded subset of $H^{s, p, \de}_0(\Om_{-\de}; \R^n)$. As this space is reflexive, there exists a $\overline{u} \in H^{s, p, \de}_0(\Om_{-\de}; \R^n)$ such that, up to a further sub-sequence, $u_{ k} \rightharpoonup \overline{u}$ weakly in $H^{s, p, \de}(\Om; \R^n)$. Due to Lemma \ref{closureOfMinimization}, these convergences are sufficient to assure that $(\overline{u}, \overline{g}) \in \mathcal{K}^{\de}_s$.
	
	Finally, we must prove that $\mathcal{F}(\overline{u}, \overline{g}) = m$. To accomplish this, note by using the compact embedding Proposition \ref{Prop: CompactEmbedding} again, we have that $u_{ k} \rightarrow \overline{u}$ strongly in $L^p(\Om; \R^n)$. Using this assertion, the weak convergence $g_{ k} \rightharpoonup \overline{g}$ in $L^{p'}(\Om; \R^n)$, and the assumptions on $\mathcal{F}$, we have that
	\begin{equation*}\label{eq: existenceMinsNLGeneralEq3}
		\mathcal{F}(\overline{u}, \overline{g}) \ \leq \ \liminf_{k \rightarrow \infty}\mathcal{F}(u_{ k}, g_{k}) \ = \ m,
	\end{equation*}
	completing the proof.
\end{proof}


\section{Asymptotic analysis}\label{Sec: asym}

One of the major questions typically considered when studying nonlocal problems is whether or not these types of problems approach their local counterparts when the nonlocality vanishes. The proper concept of convergence when dealing with variational functionals is $\Gamma$-convergence which, among other things, also provides the convergence of minimizers to minimizers.
As we will see, one of the major questions investigated in this paper is how the family of nonlocal problems relates to a corresponding local problem, as  $s \rightarrow 1^-$ or, alternatively, as $\de \rightarrow 0^+$. Before that, we carefully state the results for local problems as well, obtained analogously to their nonlocal counterparts. We also recall that our cost functional satisfies Assumption \ref{Assump: cost}; namely, the integrand $F: \R^n \times \R^n \rightarrow \R$ satisfies \ref{F1}, \ref{F2}, and \ref{F3}.

Notice, however, that given the different nature of the convergence performed in each subsection, the solutions of the corresponding local problems are defined in different domains. In particular, recalling that, for the nonlocal problem, functions are assumed to be zero in the double collar $\Om_{\de}\backslash \Om_{-\de}$, if $\delta$ is fixed, the limit functions when taking $s$ to $1$ will also vanish outside $\Om_{-\de}$. On the other hand, when we take $\de$ to $0$, the boundary where solutions vanish moves to coincide with that of $\Om$. Therefore, we define the local state equation in a general domain $\tilde{\Om}$ to be specified in each subsection. 


\subsubsection{Local convex problem}
In Subsection \ref{Sec: pLapEner}, we studied optimal control problems whose state equations were determined by a nonlocal $p$-Laplacian. Such an operator was obtained from minimizing the nonlocal $p$-Dirichlet energy, and therefore, it is expected to converge to the classical local $p$-Laplacian during the localization process. Thus, we write that local problem here. In particular, the local state equation can be written in strong form as	\begin{equation*}\label{stateEqnStrongFormL}
		\begin{cases}
			-\dive(\mathbb{A}(x)|\grad u(x)|^{p - 2}\grad u(x)) \ = \ g, \ x \in \tilde{\Om} \\
			u \ = \ 0, \ x \in \pa \tilde{\Om}.
		\end{cases}
	\end{equation*}
    Alternatively, we say that a pair $(u, g) \in W^{1,p}_0(\tilde{\Om}; \R^n) \times Z_{\text{ad}}$ verifies the  local state system in weak form if
	\begin{equation}\label{stateEqnWkFormL}
		\begin{cases}
			\mathcal{Y}_0(u, v) \ = \ \lang g, v\rang \quad \ \fa v \in W^{1, p}_0(\tilde{\Om}; \R^n)\\
			u \ = \ 0, \  x \in \pa \tilde{\Om}
		\end{cases},
	\end{equation}
where the local mapping $\mathcal{Y}_0: W^{1, p}(\tilde{\Om}; \R^n) \times W^{1, p}(\tilde{\Om}; \R^n) \rightarrow \R$ is defined as
	\begin{equation*}\label{Eq: Y0}
		\mathcal{Y}_0(u, v) \ := \ \int_{\tilde{\Om}}\mathbb{A}(x)|\grad u(x)|^{p - 2}\grad u(x):\grad v(x)dx.
	\end{equation*}

	\begin{problem}[Convex local problem]\label{lProbStatement}
		The convex local problem is to solve
		\begin{equation*}\label{stateMinL}
			\mathcal{F}(\overline{u}, \overline{g}) = \min_{(u, g) \in \Alg^{\loc}}\mathcal{F}(u, g),
		\end{equation*} 
        where the admissible class of pairs is given by
        \begin{eqnarray*}\label{localAdmiClass}
			\begin{aligned}
				\Alg^{\loc} \ := \ &\{(u, g) \in W^{1, p}_0(\tilde{\Om}; \R^n) \times Z_{\text{ad}} \ | \, (u,g) \text{ verifies } \eqref{stateEqnWkFormL}\}.
			\end{aligned}
		\end{eqnarray*}
	\end{problem}

    The analogous statement for the local existence result is presented here. Notice that in the same spirit as Lemma \ref{compactnessNLStateOperator}, the underlying local solution operator associated with the state equation \eqref{stateEqnWkFormL}, $S_{\text{loc}}:L^{p'}(\tilde{\Om};\Rn) \to W^{1,p}_0(\tilde{\Om};\Rn) \hookrightarrow L^p(\tilde{\Om}; \R^n) $ is compact. Furthermore, $\mathcal{A}^{\loc}$ is convex as well. Therefore, by Theorem \ref{Th: abstractMinimizationConvex} we have the following local result.
\begin{theorem}[Well-posedness of local control problem]\label{WPLControlThm}
	There exists a pair $(\overline{u}, \overline{g}) \in W^{1, p}_0(\tilde{\Om}; \R^n) \times Z_{\text{ad}}$ solving Problem \ref{lProbStatement}. 
    Furthermore, this solution is unique.  
\end{theorem}


\subsubsection{Local non-convex problem}\label{Subsec: localNonconvex}
 
	Now we present the local analogue of Problem \ref{nlProbStatementG}. First, 	we introduce the analogous local energy, $\mathcal{W}^{\loc}_g: L^p(\R^n; \R^n) \rightarrow \R \cup \{+\infty\}$, for $g\in Z_{\text{ad}}$, defined as
	\begin{equation*}\label{Eq: GenEnergyL}
		\mathcal{W}^{\loc}_g(u)\ := \ \begin{cases}
			\int_{\tilde{\Om}}W(x, u(x), \grad u(x))dx - \lang g, u\rang, \ u \in W^{1, p}(\tilde{\Om}; \R^n) \\
			+\infty  \ \text{otherwise},
		\end{cases}
	\end{equation*}
    where $W$ is given as in Subsection \ref{subsec: convexity}.

	\begin{problem}[Non-convex local problem]\label{lProbStatementG}
		Our local non-convex problem is to find
		\begin{equation*}\label{stateMinLG}
			\mathcal{F}(\overline{u}, \overline{g}) \ = \ \min_{(u, g)
				\in \mathcal{K}^{\loc}}\mathcal{F}(u, g),
		\end{equation*} 
        where the admissible class of pairs is given by
    \begin{equation*}\label{localAdmiClassG}
			\mathcal{K}^{\loc} \ := \ \{(u, g) \in W^{1, p}_0(\tilde{\Om}; \R^n) \times Z_{\text{ad}} \ | \ u \in  \argmin_{v \in W^{1, p}_0(\tilde{\Om}; \R^n)}\mathcal{W}^{\loc}_g(v)\}.
		\end{equation*}

		\end{problem}
By reproducing the same steps as in subsection \ref{Sec: nonconvex} we obtain the following result.

\begin{theorem}[Existence of solutions to non-convex local control problem]
	Let $W: \tilde{\Om} \times \R^n \times \R^{n \times n} \rightarrow \R$ be an energy density such that $\mathcal{W}^{\loc}_{g}(\cdot)$ is not identically infinity in $W^{1, p}_0(\tilde{\Om}; \R^n)$. Assume, in addition, one of the following: 
    \begin{itemize}
        \item[a)] $W$ is polyconvex and \ref{H1} is satisfied with $p\geq n-1$ and $q\geq \frac{n}{n-1}$.
        \item[b)] $W$ is quasiconvex and \ref{H2} is satisfied.
    \end{itemize}
  Then,  Problem \ref{lProbStatementG} has a minimizer $(\overline{u}, \overline{g}) \in \mathcal{K}^{\text{loc}}$.
\end{theorem}
In particular, by repeating the proof of Lemma \ref{bddAdmissibleState}, we have the boundedness of admissible states, i.e., that the set
\begin{equation*}\label{eq:ULAdmis}
	U_{\loc} \ := \ \{u \in W^{1, p}_{0}(\tilde{\Om}; \R^n), \ \ex g \in Z_{\text{ad}}, \ \mathcal{W}^{\loc}_g(u) \leq \mathcal{W}^{\loc}_g(v) \quad \ \fa v \in W^{1, p}_0(\tilde{\Om}; \R^n)\} ,
\end{equation*}
is a bounded subset of $W^{1, p}_0(\tilde{\Om}; \R^n)$.


\subsection{Convergence results as $s \rightarrow 1^-$}\label{subsec: s->1pLap}

For this subsection, which is independent of Subsection \ref{subsec: de->0pLap}, we assume that $\de > 0$ is fixed. In this case we specify two extra hypothesis, the first one of which concerns the kernel and is needed to invoke case \emph{a)} of the Nonlocal Poincaré Inequality of Proposition \ref{Prop: NLPoincare}.
\begin{itemize}
      \item\label{As1} Let $w_\de(0)=a_0=1$. 
      \item Let $\tilde{\Om}=\Om_{-{\de}}$.
\end{itemize}



\subsubsection{$\Gamma$-convergence of the energies}

First, we recall a special case of \cite[Theorem 3]{cueto2023variational},  which indicates strong convergence of nonlocal gradients.

	\begin{lemma} \label{Thm: gradUnifConv}
		Let $u \in W^{1, p}(\Om; \R^n)$ with $u=0$ in $\Om_{\de}\backslash\Om_{-\de}$. Then, $D^{s_j}_{\de}u \rightarrow \grad u$ strongly in $L^p(\Om; \R^{n\times n})$ for every sequence $\{s_j\}^{\infty}_{j = 1} \subset [0, 1]$, $s_j \rightarrow 1$, as $j \rightarrow \infty$. 
	\end{lemma}
            
We will also need an additional compactness property with a varying fractional parameter, originally proven in \cite[Lemma 9]{cueto2023variational}. Note that, as stated in that paper, an assumption of $\Om_{-\de}$ being Lipschitz was included, but that is not needed in our setting since we define the $\|\cdot\|_{H^{s, p, \de}(\Om; \R^n)}$ norm by density rather than by the distributional definition of $D^s_{\de}$.
	
	\begin{proposition}[Weak compactness with variable fractional exponent]\label{Prop: CompactSVary}
		Suppose $\{s_j\}^{\infty}_{j = 1} \subset [0, 1]$ is a sequence such that $s_j \rightarrow 1$ where $u_j \in H^{s_j, p, \de}_0(\Om_{-\de}; \R^n)$ is a sequence such that
		\begin{equation*}\label{Eq: CompactSVary}
			\sup_{j \in \N^+}\|D^{s_j}_{\de}u_j\|_{L^p(\Om; \R^{n\times n})} \ < \ \infty.
		\end{equation*}
		Then there exists $u \in W^{1, p}(\Om_{-\de}; \R^n)$, $u=0$ in $\Om_{\de}\backslash\Om_{-\de}$ so that, up to a non relabeled sub-sequence, $u_j \rightarrow u$ strongly in $L^p(\Om_{\de}; \R^n)$  and $D^{s_j}_{\de}u_j \rightharpoonup \grad u$ weakly in $L^p(\Om; \R^{n\times n})$.
	\end{proposition}

The next $\Ga$-convergence result of the energy functionals is actually derived from the one proven in \cite[Theorem 7]{cueto2023variational}. Indeed, even though \cite[Theorem 7]{cueto2023variational} does not include $u$-dependence, this does not play a role in the proof, which is based on the translation results (from nonlocal to local gradients) derived in \cite{cueto2023variational} and the characterization of quasiconvexity in this framework (Theorem 5 of the same paper), which does include the $u$-dependence. Therefore the very same steps of that proof also apply for energy densities exhibiting the $u$-dependence. Similarly to \ref{minimizationOfNLEnergy}, these energy densities verify the required hypotheses. While the proof of \cite[Theorem 7]{cueto2023variational} does not explicitly show a $u$-depenedence, we may repeat the argument of that proof because the weak lower semicontinuity result\cite[Theorem 8.11]{dacorogna2007direct} can be used to allow for a $u$-dependence. Additional references include \cite[Theorem 7.5]{fonseca2007modern} or \cite[Theorem 5.4]{ball1981null} for convex energy density functions, and  \cite[Theorem 5.20]{Rin} for quasiconvex energy density functions with a stronger growth condition.


\begin{theorem}[$\Ga$-convergence as $s \rightarrow 1^-$]\label{Gas->1Theorem}
	Let $g \in Z_{\text{ad}}$ be fixed and $W$ be a quasiconvex energy density verifying the hypothesis of Subsection \ref{Sec: nonconvex} and, in particular \ref{H2}. Then, the family of functionals $\{\mathcal{W}^{\de, s}_g\}_{s < 1}$ will $\Ga$-converge in the strong $L^p(\Om; \R^n)$ topology to $\mathcal{W}^{\loc}_g$, which we will denote $\mathcal{W}^{\de, s}_g \xrightarrow{\Ga, \ s \rightarrow 1^-} \mathcal{W}^{\loc}_g$. In other words, we have the following:
	\begin{enumerate}[label=\textbf{GCs\arabic*}]
		\item\label{GCs1} If $\{u_s\}_{s < 1} \subset L^p(\Om; \R^n)$ is a sequence such that $u_s \rightarrow u$ strongly in $L^p(\Om; \R^n)$, then we have the \textbf{lim-inf inequality}
		\begin{equation*}\label{Gas->1Liminf}
			\mathcal{W}^{\loc}_g(u) \ \leq \ \liminf_{s \rightarrow 1^-}\mathcal{W}^{\de, s}_g(u_s).
		\end{equation*}
		\item\label{GCs2} If $u \in L^p(\Om; \R^n)$, then there exists a \textbf{recovery sequence} of $\{u_s\}_{s < 1} \subset L^p(\Om; \R^n)$ such that $u_s \rightarrow u$ strongly in $L^p(\Om; \R^n)$ and
		\begin{equation*}\label{Gas->1Limsup}
			\mathcal{W}^{\loc}_g(u) \ \geq \ \limsup_{s \rightarrow 1^-} \mathcal{W}^{\de, s}_g(u_s).
		\end{equation*}
	\end{enumerate}
\end{theorem}

\begin{remark}\label{Rmk: Mosco}
	The aforementioned argument also demonstrates that $\mathcal{W}^{\de, s}_g \xrightharpoonup{\Ga, \ s \rightarrow 1^-} \mathcal{W}^{\loc}_g$ in the weak topology on $L^p(\Om; \R^n)$. Hence in this setting our energies exhibit \textbf{Mosco convergence}, which says that the lim-inf inequality holds with respect to the weak topology, whereas the lim-sup inequality holds with respect to the strong topology. See \cite{gounoue2020mosco} for an introduction to Mosco convergence in the context of a different class of nonlocal energies.
\end{remark}


\subsubsection{Convergence of states and controls}

Now we may obtain the convergence of minimizers by following the proof of \cite[Theorem 5.7]{mengesha2023control}. The first step in this direction is a compactness lemma, akin to \cite[Lemma 5.1]{mengesha2023control}. 
\begin{lemma}[Compactness of controls and states, $s \rightarrow 1^-$]\label{compactnessSolns} Let $W$ be an energy density verifying the conditions of Theorem \ref{Gas->1Theorem} and
	$\{(\overline{u_s}, \overline{g_s})\}_{s < 1}$ denote a family of solutions to Problem \ref{nlProbStatementG}. Then there exists a sequence $\{s_j\}^{\infty}_{j = 1}$ such that $s_j \rightarrow 1$, and a pair $(\widehat{u}, \widehat{g}) \in \mathcal{K}^{\loc}$ such that $\overline{u_{s_j}} \rightarrow \widehat{u}$ strongly in $L^p(\Om_{\de}; \R^n)$ and $\overline{g_{s_j}} \rightharpoonup \widehat{g}$ weakly in $L^{p'}(\Om; \R^n)$, as $j \rightarrow \infty$. 
\end{lemma}

\begin{proof}
	Since $\{\overline{g_s}\}_{s < 1}$ belongs to $Z_{\text{ad}}$, a bounded and weakly closed subset of $L^{p'}(\Om; \R^n)$, existence of a function $\widehat{g} \in Z_{\text{ad}}$ such that $\overline{g_s} \rightharpoonup \widehat{g}$ weakly in $L^{p'}(\Om; \R^n)$ immediately follows from the reflexivity of $L^{p'}(\Om; \R^n)$ and weak closedness of $Z_{\text{ad}}$.
    
	As for the states, we repeat the steps of the proof of  Lemma \ref{bddAdmissibleState} to obtain
    \begin{equation*}
             \|D^{s_j}_{\de}\overline{u_{s_j}}\|^p_{L^p(\Om; \R^{n \times n})} \ \leq \ \|\overline{g_{s_j}}\|_{L^{p'}(\Om; \R^n)}\|\overline{u_{s_j}}\|_{L^p(\Om; \R^n)} +\|a\|_{L^1(\Om)} +\left| W_0^{\de,s_j}(0)\right|.
       \end{equation*}    
       
       By \emph{\ref{H2}}, Theorem \ref{Thm: gradUnifConv} and the Dominated Convergence Theorem we have that the term $\left| W_0^{\de,s_j}(0)\right|$ is bounded in $j$.  Using the Nonlocal Poincaré inequality Proposition \ref{Prop: NLPoincare} \emph{a)}, (see also \cite[Theorem 4]{cueto2023variational}), together with the uniform bound from the definition of $Z_{\text{ad}}$ leads to
        \begin{equation*}
               \|D^{s_j}_{\de}\overline{u_{s_j}}\|^p_{L^p(\Om; \R^{n \times n})} \ \leq  \ C\|D^{s_j}_{\de}\overline{u_{s_j}}\|_{L^p(\Om; \R^{n \times n})} +\|a\|_{L^1(\Om)}+\bar{C},
        \end{equation*}
    where $\bar{C},C>0$ are constants independent of $j$, from which it follows that $\{[\overline{u_{s_j}}]_{H^{s_j, p, \de}(\Om; \R^n)}\}^{\infty}_{j = 1}$ is a bounded sequence (of real numbers). Then, due to the compactness result Proposition \ref{Prop: CompactSVary} (see also \cite[Lemma 9]{cueto2023variational}), there exists a (not relabeled) sub-sequence and a function $\widehat{u} \in W^{1, p}(\Om_{-\de}; \R^n)$, verifying, in addition, that $u=0$ in $\Om_{\de}\backslash\Om_{-\de}$, such that $\overline{u_{s_j}} \rightarrow \widehat{u}$ strongly in $L^p(\Om_{\de}; \R^n)$ as $j \rightarrow \infty$. 
 Moreover, due to the $\Ga$-convergence result Theorem \ref{Gas->1Theorem}, and observing that $\{\mathcal{W}^{\de, s}_0\}_{s < 1}$ is an equi-coercive family of functionals with respect to strong convergence in $L^p(\Om; \R^n)$, we have that minimizers of the family $\{\mathcal{W}^{\de, s}_0\}_{s < 1}$ converge to minimizers of the local energy $\mathcal{W}^{\loc}_0$. Furthermore, since $\mathcal{W}^{\de, s}_{\overline{g_{s}}}$ and $\mathcal{W}^{\loc}_{\overline{g}}$ serve as continuous perturbations of $\mathcal{W}^{\de, s}_0$ and $\mathcal{W}^{\loc}_0$, respectively, this property extends to these energies as well; in particular, $(\widehat{u}, \widehat{g}) \in \mathcal{K}^{\loc}$.
    
\end{proof}
Next we have the convergence of minimizers of the nonlocal optimal control problems to their local counterpart. However, even though the previous result holds for the non-convex case, which implies that the limit of solutions of the nonlocal optimal control problems are admissible pairs of the local one, given the lack of uniqueness of the non-convex state-equation we are only able to provide this result for the convex case. 
We will make use of the uniqueness result Theorem \ref{WPLControlThm}, which holds whenever the cost functional is strictly convex.

\begin{theorem}[Convergence of minimizers as $s \rightarrow 1^-$]\label{convMinss->1}   
Let $W$ be an energy density verifying the conditions of Theorem \ref{Gas->1Theorem} and suppose $F$ is convex in its second argument. Let $(\overline{u_s}, \overline{g_s})$ denote the solution of Problem \ref{nlProbStatement}, while $(\overline{u}, \overline{g})$ denotes the solution of Problem \ref{lProbStatement}. Then, as $s \rightarrow 1^-$, we have that 
\begin{itemize}
    \item[a)]$\overline{u_s} \rightarrow \overline{u}$ strongly in $L^p(\Om_{\de}; \R^n)$,
    \item[b)] $\overline{g_s} \rightharpoonup \overline{g}$ weakly in $L^{p'}(\Om; \R^n)$,
    \item[c)] $D^s_\de \bar{u}_s \to \nabla \bar{u}$ strongly in $L^p(\Om; \R^{n \times n}).$
\end{itemize}
Moreover, we have the limit 
	\begin{equation}\label{Eq: lims->1MinCon}
		\lim_{s \rightarrow 1^-}\mathcal{F}(\overline{u_s}, \overline{g_s}) \ = \ \mathcal{F}(\overline{u}, \overline{g}).
	\end{equation}
\end{theorem}

\begin{proof} We first prove \emph{a)} and \emph{b)}. Due to Lemma \ref{compactnessSolns}, there exists a sequence $\{s_j\}^{\infty}_{j = 1}$ such that $s_j \rightarrow 1$, and a pair $(\widehat{u}, \widehat{g}) \in \mathcal{K}^{\loc}$ such that $\overline{u_{s_j}} \rightarrow \widehat{u}$ strongly in $L^p(\Om_{\de}; \R^n)$ and $\overline{g_{s}} \rightharpoonup \widehat{g}$ weakly in $L^{p'}(\Om; \R^n)$, as $s \rightarrow 1^-$. We will see that $(\widehat{u}, \widehat{g})$ is actually a minimizer of Problem \ref{lProbStatement} which, by uniqueness, coincides with $(\overline{u}, \overline{g})$. 
First we observe that every element of $\mathcal{K}^{\loc}$ can be obtained as a limit of solutions of the nonlocal problems. Indeed, given $(v, f) \in \mathcal{K}^{\loc}$ be arbitrary, and let $v_s$ denote the solution to \eqref{nonlocalAdmiClassG} with right-hand side $f$, if we repeat the argument of Lemma \ref{compactnessSolns} with $\overline{g_{s_j}} := f$ for all $s < 1$, what we obtain is a function $\widehat{v} \in W^{1, p}(\Om_{-\de}; \R^n)$ such that $\hat{v}=0$ in $\Om_{\de}\backslash\Om_{-\de}$ and $v_s \rightarrow \widehat{v}$ strongly in  $L^p(\Om_{\de}; \R^n)$. Since the local state equation is well-posed, we have that $\widehat{v} = v$.
	Then, by the Dominated Convergence Theorem (which applies due to the growth condition \eqref{Gqcoercive}),
	\begin{equation}\label{convMinss->1Eq1}   
		\lim_{s \rightarrow 1^-}\mathcal{F}(v_s, f) \ = \ \mathcal{F}(v, f).
	\end{equation}
	Meanwhile, due to the minimality of $(\overline{u_s}, \overline{g_s})$ for each $s < 1$, we have the inequality
	\begin{equation}\label{convMinss->1Eq2} 
		\liminf_{s \rightarrow 1^-}\mathcal{F}(v_s, f) \ \geq \ \liminf_{s \rightarrow 1^-}\mathcal{F}(\overline{u_s}, \overline{g_s}).
	\end{equation}
	Next, we may use the weak lower semi-continuity of the $L^{p'}$ norm and the Dominated Convergence Theorem to see that
	\begin{equation}\label{convMinss->1Eq3} 
		\liminf_{s \rightarrow 1^-}\mathcal{F}(\overline{u_s}, \overline{g_s}) \ \geq \ \mathcal{F}(\overline{u}, \overline{g}).
	\end{equation}
	The conclusion of this inequality chain is that $\mathcal{F}(v, f) \geq \mathcal{F}(\overline{u}, \overline{g})$. Moreover, setting $v := \overline{u}$ and $f := \overline{g}$ in \eqref{convMinss->1Eq1} yields the limit \eqref{Eq: lims->1MinCon}.
    
It remains to prove \emph{c)}, i.e., that  $D^s_{\de}\overline{u_s} \rightarrow \grad \overline{u}$ strongly in $L^p(\Om; \R^{n\times n})$. 
	Since the $\overline{u_s}$ are optimal, they verify the state equation. In particular, taking limits we have that 
	\begin{equation}\label{eq: strongConvMinMul0}
		\lim_{s \rightarrow 1^-}\int_{\Om} \mathbb{A}(x)|D^s_\de \overline{u_s}(x)|^{p - 2}D^s_{\de}\overline{u_s}(x):D^s_{\de}\overline{u_s}(x)\, dx \ = \lim_{s \rightarrow 1^-}\ \int_{\Om} \overline{u_s} (x) \cdot \overline{g_s}(x)\, dx.
	\end{equation}
	Now, since $\overline{u_s}$ converges strongly in $L^p(\Om; \R^{n})$ and $\overline{g_s}$ converges weakly in $L^p(\Om; \R^{n})$ we have that
	\begin{eqnarray}\label{eq: strongConvMinMul}
		\begin{aligned}
		& \lim_{s \rightarrow 1^-}\int_{\Om} \overline{u_s} (x) \cdot \overline{g_s}(x)\, dx=\int_{\Om} \overline{u}(x) \cdot \overline{g}(x)\,dx
			\\&=\ \int_{\Om}\mathbb{A}(x)|\grad \overline{u}(x)|^{p - 2}\grad \overline{u}(x) : \grad \overline{u}(x) dx,
		\end{aligned}
	\end{eqnarray}
	where the last equality is due to $\overline{u}$ being an optimal state for $\mathcal{W}^{\loc}_{\overline{g}}(\cdot)$.
	Since $\mathbb{A} \in L^{\infty}(\Om; \R^{n\times n})$, and $D^s_\de \bar{u}_s \rightharpoonup \nabla \bar{u}$ (Proposition \ref{Prop: CompactSVary}) we have that $\mathbb{A}D^s_\de\bar{u}_s \rightharpoonup \mathbb{A}\nabla \bar{u}$ in $L^p(\Om, \R^{n \times n})$. This weak convergence in combination with the convergence of the weighted norms from \eqref{eq: strongConvMinMul0} and \eqref{eq: strongConvMinMul} implies
    \begin{equation*}
       \lim_{s \to 1^-} \|D^s_\de \bar{u}_s -\nabla \bar{u}\|_{L^p_{\mathbb{A}}(\Om,\R^{n \times n})}=0,
    \end{equation*}
    where for $A\in L^p(\Om, \R^{n \times n})$, we have defined the weighted norm
    \begin{equation*}
         \|A\|_{L^p_{\mathbb{A}}(\Om,\R^{n \times n})}:=\int_{\Om}\mathbb{A}(x)|A(x)|^{p - 2}A(x) : A(x) dx.
    \end{equation*}
    Finally, since $\mathbb{A}$ is uniformly elliptic in $\Om$ (see \eqref{Eq: UnifEllip}), 
    \begin{equation*}
        \lim_{s \to 1^{-}} \|D^s_\de \bar{u}_s -\nabla \bar{u}\|_{L^p(\Om,\R^{n \times n})} \leq \lim_{s \to 1^{-}} \frac{1}{\mu}\|D^s_\de \bar{u}_s -\nabla \bar{u}\|_{L^p_{\mathbb{A}}(\Om,\R^{n \times n})}=0,
    \end{equation*}
    concluding the proof.
	%
\end{proof}
Given the lack of uniqueness of the non-convex state equation, the obstacle for obtaining a similar result lies in comparing the minimality of the limit pair $(\widehat{u},\widehat{g})$ against every pair in $\mathcal{K}^{\loc}$ since now there is no guarantee that every such pair can be obtained as a limit of solutions of the nonlocal problems.
While we cannot obtain a full convergence of minimizers result similar to Theorem \ref{convMinss->1}, we can obtain a weaker result that involves comparing a candidate limiting pair to some other locally admissible pair for each control in $Z_{\text{ad}}$.
\begin{corollary}[Relative minimization]\label{cor: compactnessRecminsQuasi}
	In the context of Theorem \ref{convMinss->1}, let $\{(\overline{u_s}, \overline{g_s})\}_{s < 1}$ be a family of solutions for Problem \ref{nlProbStatementG}. Then there exists a pair $(\overline{u}, \overline{g}) \in \mathcal{K}^{\loc}$ such that, up to a sub-sequence, $\overline{u_s} \rightarrow \overline{u}$ strongly in $L^p(\Om_{\de}; \R^n)$, $\overline{g_s} \rightharpoonup \overline{g}$ weakly in $L^{p'}(\Om; \R^n)$, and for every $f \in Z_{\text{ad}}$, there exists a $\hat{u} \in W^{1, p}(\Om_{-\de}; \R^n)$ so that $\hat{u}=0$ in $\Om_{\de}\backslash\Om_{-\de}$, $(\hat{u}, f) \in \mathcal{K}^{\loc}$ and
	\begin{equation}\label{lem: compactnessRecminsQuasiIneq}
		\mathcal{F}(\overline{u}, \overline{g}) \ \leq \ \mathcal{F}(\hat{u}, f).
	\end{equation}
\end{corollary}
\begin{proof}
    Given the lack of uniqueness of the non-convex state equation we cannot identify (before \eqref{convMinss->1Eq1}) $\hat{v}$ with $v$ in the proof of Theorem \ref{convMinss->1}. Thus, combining \eqref{convMinss->1Eq1},\eqref{convMinss->1Eq2} and \eqref{convMinss->1Eq3}  leads to \eqref{lem: compactnessRecminsQuasiIneq}.
\end{proof}
In addition, just as in the setting of the convex control problems, we have the lim-inf inequality
\begin{equation*}\label{eq: liminfstillholds}
	\liminf_{s \rightarrow 1^-}\mathcal{F}(\overline{u_s}, \overline{g_s}) \ \geq \ \mathcal{F}(\overline{u}, \overline{g}).
\end{equation*}

Finally, we have that the convergence of controls can be improved to strong convergence. Namely, while the weak $L^{p'}(\Om; \R^n)$ topology is the natural topology to consider for the controls in the context of $\Ga$-convergence, it turns out that, using the structure of the cost functional, we may improve this convergence to strong convergence in any $L^r(\Om; \R^n)$ space (for $r \in [1, \infty)$). 
\begin{theorem}[Strong convergence of controls as $s \rightarrow 1^-$]\label{Th: StrConvControls} Let $W$ be an energy density verifying the conditions of Theorem \ref{Gas->1Theorem}. 
Assume $\{(\overline{u_{ s}}, \overline{g_{s}})\})_{s < 1}$ is a sequence of solution to Problem \ref{nlProbStatement} (respectively Problem \eqref{nlProbStatementG}), converging to $(\overline{u}, \overline{g})$, a solution of the respective local problem. 
    Then, as $s \rightarrow 1^-$, we have that $\overline{g_{s}} \rightarrow \overline{g}$ strongly in $L^r(\Om; \R^n)$ for any $r \in [1, \infty)$.
\end{theorem}

\begin{proof}
	The main step is proving that $\overline{g_{ s}} \rightarrow \overline{g}$ strongly in $L^{p'}(\Om; \R^n)$. We proved in the course of Theorem \ref{convMinss->1} that $\overline{u_{s}} \rightarrow \overline{u}$ strongly in $L^p(\Om; \R^n)$. Due to the growth bound \eqref{Gqcoercive} on $F$ and the Generalized Dominated Convergence Theorem, we have that
	\begin{equation}\label{Eq: StrConvControls1}
		\lim_{s \rightarrow 1^-}\int_{\Om}F(x, \overline{u_{ s}}(x))dx \ = \ \int_{\Om}F(x, \overline{u}(x))dx.
	\end{equation}
	The equation \eqref{Eq: StrConvControls1} along with \eqref{Eq: lims->1MinCon} tells us that
	\begin{equation}\label{Eq: StrConvControls2}
		\lim_{s \rightarrow 1^-}\int_{\Om}\La(x)|\overline{g_{ s}}(x)|^{p'}dx \ = \ \int_{\Om}\La(x)|\overline{g}(x)|^{p'}dx.
	\end{equation}
	Now notice that the weighted quantity $\|\cdot\|_{L^{p'}_{\La}(\Om; \R^n)}: L^{p'}(\Om; \R^n) \rightarrow [0, \infty)$
	\begin{equation*}\label{Eq: StrConvControls3}
		\|g\|_{L^{p'}_{\La}(\Om; \R^n)} \ := \ \left(\int_{\Om}\La(x)|g(x)|^{p'}dx\right)^{\frac{1}{p'}}
	\end{equation*}
	defines a norm on $L^{p'}(\Om; \R^n)$ that is equivalent to the standard $\|\cdot\|_{L^{p'}(\Om; \R^n)}$ norm, owing to $\La$ being positive and bounded. Thus \eqref{Eq: StrConvControls2} can be written as
	\begin{equation}\label{Eq: StrConvControls4}
		\lim_{s \rightarrow 1^-}\|\overline{g_{ s}}\|^{p'}_{L^{p'}_{\La}(\Om; \R^n)} \ = \ \|\overline{g}\|^{p'}_{L^{p'}_{\La}(\Om; \R^n)}.
	\end{equation}
	Now, since  $\overline{g_{s}} \rightharpoonup \overline{g}$ weakly in $L^{p'}(\Om; \R^n)$, and $\La \in L^{\infty}(\Om) \subset L^p(\Om)$, we have that  $\La\overline{g_{s}} \rightharpoonup \La\overline{g}$ weakly in $L^{p'}(\Om; \R^n)$ as well. Since $L^{p'}(\Om; \R^n)$ is reflexive, this observation combined with \eqref{Eq: StrConvControls4} tells us that
	\begin{equation*}\label{Eq: StrConvControls5}
		\lim_{s \rightarrow 1^-}\|\overline{g_{ s}} - \overline{g}\|_{L^{p'}_{\La}(\Om; \R^n)} \ = \ 0.
	\end{equation*}
	From this and Assumption \ref{Assump: cost} there exists $0<\la\leq \La(x)$ such that
	\begin{equation*}\label{Eq: StrConvControls6}
		\lim_{s \rightarrow 1^-}\|\overline{g_{ s}} - \overline{g}\|_{L^{p'}(\Om; \R^n)} \ \leq \ \lim_{s \to 1^-}\frac{1}{\la}\|\overline{g_{ s}} - \overline{g}\|_{L^{p'}_{\La}(\Om; \R^n)} \ = \ 0,
	\end{equation*}
	completing the proof that $\overline{g_{s}} \rightarrow \overline{g}$ strongly in $L^{p'}(\Om; \R^n)$. Now, this immediately implies $\overline{g_{ s}} \rightarrow \overline{g}$ strongly in $L^r(\Om; \R^n)$ for any $r \in [1, p')$ by Hölder's inequality. 
	
	Finally we demonstrate also that $\overline{g_{s}} \rightarrow \overline{g}$ strongly in $L^r(\Om; \R^n)$ for $r \in (p', \infty)$. Due to the construction of the set $Z_{\text{ad}}$ from \eqref{Eq: linearAdSet}, we have the estimate
	\begin{eqnarray*}\label{Eq: StrConvControls7}
		\begin{aligned}
			\|\overline{g_s} - \overline{g}\|^r_{L^r(\Om; \R^n)} \ &= \ \int_{\Om}|\overline{g_{s}}(x) - \overline{g}(x)|^{r - p'}|\overline{g_{s}}(x) - \overline{g}(x)|^{p'}dx \\
			&\leq \ \max\{\|a\|_{L^{\infty}(\Om; \R^n)}, \|b\|_{L^{\infty}(\Om; \R^n)}\}^{r - p'}\|\overline{g_{s}} - \overline{g}\|^{p'}_{L^{p'}(\Om; \R^n)} \rightarrow 0,
		\end{aligned}
	\end{eqnarray*}
	which completes the proof for all $r \in [1, \infty)$.

\end{proof}

\begin{remark}\label{Rmk: roleOfPenaltyTerm}
	Another important matter in control problems, that we do not dwell on in the present paper, is the role of the penalty term for regularizing the optimal control problems. Another advantage of having the penalty term is the ability to carry out this very proof.
\end{remark}

\begin{remark}\label{Rmk: GammaConvergenceCostFunctionals}
	In effect, Theorem \ref{convMinss->1} is proven by establishing the $\Ga$-convergence of the sequence of nonlocal reduced cost functionals, defined by \eqref{NLReducedCostFunctional} for each $\de > 0$ and $s \in (0, 1)$ in the $L^{p'}(\Om; \R^n)_{\om}$ topology, as $s \rightarrow 1^-$, to the local reduced cost functional 
	\begin{equation*}\label{Eq: lReducedCost}
		j^{\loc}(g) \ := \ \int_{\Om}F(x, S^{\loc}g(x))dx + \int_{\Om}\La(x)|g(x)|^{p'}dx.
	\end{equation*}
	Here $S^{\loc}: L^{p'}(\Om; \R^n) \rightarrow L^p(\Om; \R^n)$ is the local solution mapping induced by the well-posedness of the local state equation \eqref{stateEqnWkFormL}. However, thanks to Theorem \ref{Th: StrConvControls}, we actually strengthen this $\Ga$-convergence to hold on the strong $L^r(\Om; \R^n)$ topology, for any $r \in [1, \infty]$. 
    
\end{remark}


\subsection{Convergence results as $\de \rightarrow 0^+$}\label{subsec: de->0pLap}

Now we state the analogous convergence results when $s \in (0, 1)$ is fixed and $\de \rightarrow 0^+$. These results mirror those presented in Section \ref{subsec: s->1pLap}. Without loss of generality, we consider the initial $\de_0$ as $\de_0=1$. Therefore, for every $\de \in (0,1)$, $\Om_{\de}\subset \Om_1:=\Om+B(0,1)$. Now, we specify the next two extra hypothesis, the first one of which concerns the kernel and is needed to invoke case \emph{b)} of the Nonlocal Poincaré Inequality of Proposition \ref{Prop: NLPoincare}, whereas the second one stands for the domain of definition of the corresponding local problem.
\begin{itemize}
      \item Let $\rho_\de^s(x)=\de^{-n}\rho^s_1 \left(\frac{x}{\de}\right)$, with $\int_{B(0,1)}\rho^s_1(x)\, dx=n$.
      \item Let $\tilde{\Om}=\Om$.
\end{itemize}


\subsubsection{$\Gamma$-convergence of the energies}

In this section we address the $\Gamma$-converge of the energies, which will imply, in turn, that the limit of solutions of nonlocal problems are admissible solutions of the local one. First, we recall a special case of \cite[Lemma 3.1]{cueto2025gamma}, which provides the strong convergence of nonlocal gradients when $\de$ vanishes.

\begin{lemma} \label{le: localication of the nl gradient when de goes to 0}
    For each $u \in W^{1,p}(\Om; \R^n)$, it holds that $\chi_{\Om_{-\de}}D^s_\de u \to \nabla u$ in $L^p(\Om;\R^{n \times n})$ as $\de \to 0$, where $\chi_{\Om_{-\de}}$ stands for the characteristic function. 
\end{lemma}

This convergence of nonlocal gradients led to a compactness result for sequences of functions with vanishing horizon, \cite[Lemma 3.6]{cueto2025gamma}, which we recall here.
\begin{proposition}[Convergent sub-sequences for vanishing horizon]\label{le:compactness} 
	Let $\{\de_j\}_{j\in \mathbb{N}} \subset (0,1]$ be a sequence with $\de_j \to 0$ and suppose that $u_j \in H^{s,p,\de_j}_0(\Omega_{-\de_j}; \R^n)$ for each $j \in \N$ with
	\[
	\sup_{j \in \N} \|D^s_{\de_j} u_j\|_{L^p({\Omega};\R^{n \times n})} < \infty.
	\]
	Then, there is a $u \in W^{1,p}_0(\Omega; \R^n)$ (extended outside of $\Om$ as zero) such that, up to a (not relabeled) subsequence, 
	\[
	u_j \to u \ \  \text{in $L^p(\R^n; \R^n)$} \quad \text{and} \quad D^s_{\de_j}u_j \rightharpoonup \nabla u \ \  \text{in $L^p(\R^n;\R^{n \times n})$ as $j \to \infty$.}
	\]
\end{proposition}
This $\Ga$- convergence result is derived from \cite[Theorem 3.7]{cueto2025gamma}: \cite[Remark 3.8 and Example 3.9]{cueto2025gamma} show that it applies to our kernel, with the corresponding Poincaré inequality and compactness result (\cite[Lemma 3.6]{cueto2025gamma}) for varying $\de$. Arguing similarly to the $\Ga$-convergence results on $s$ from Section \ref{subsec: s->1pLap}, references such as \cite[Theorem 8.11]{dacorogna2007direct} can be employed to ensure that the result holds despite including $u$-dependence, as it does not play a role in the proof.

\begin{theorem}[$\Ga$-convergence as $\de \rightarrow 0^+$]\label{Gade->0: GaConv}
	Let $g \in Z_{\text{ad}}$ be fixed and $W$ be a quasiconvex energy density verifying the hypothesis of Subsection \ref{Sec: nonconvex} and, in particular \ref{H2}. Then, the family of functionals $\{\mathcal{W}^{\de, s}_g\}_{s < 1}$ will $\Ga$-converge in the strong $L^p(\Om; \R^n)$ topology to $\mathcal{W}^{\loc}_g$, which we will denote $\mathcal{W}^{\de, s}_g \xrightarrow{\Ga, \ \de \rightarrow 0^+} \mathcal{W}^{\loc}_g$. In other words, we have the following:
	\begin{enumerate}[label=\textbf{GCs\arabic*}]
		\item\label{GCd1} If $\{u_{\de}\}_{\de > 0} \subset L^p(\Om; \R^n)$ is a sequence such that $u_{\de} \rightarrow u$ strongly in $L^p(\Om; \R^n)$, then we have the \textbf{lim-inf inequality}
		\begin{equation*}\label{Gade->0Liminf}
			\mathcal{W}^{\loc}_g(u) \ \leq \ \liminf_{\de \rightarrow 0^+}\mathcal{W}^{\de, s}_g(u_{\de}).
		\end{equation*}
		\item\label{GCd2} If $u \in L^p(\Om; \R^n)$, then there exists a \textbf{recovery sequence} of $\{u_{\de}\}_{\de > 0} \subset L^p(\Om; \R^n)$ such that $u_{\de} \rightarrow u$ strongly in $L^p(\Om; \R^n)$ and
		\begin{equation*}\label{Gade->0Limsup}
			\mathcal{W}^{\loc}_g(u) \ \geq \ \limsup_{\de \rightarrow 0^+} \mathcal{W}^{\de, s}_g(u_{\de}).
		\end{equation*}
	\end{enumerate}
\end{theorem}


\subsubsection{Convergence of states and controls}

Before addressing the convergence of states we need the following compactness result. In particular, it states that the limit of solutions of the nonlocal optimal control problem return a solution to the local state equation. We provide the proof despite its similarity to that of Lemma \ref{compactnessSolns} due to the presence of a vanishing nonlocal boundary.

\begin{lemma}[Compactness of controls and states, $\de \rightarrow 0^+$]\label{De->0: compactnessSolns}
	Let $W$ be an energy density verifying the conditions of Theorem \ref{Gade->0: GaConv}  and $\{(\overline{u_{\de}}, \overline{g_{\de}})\}_{\de > 0}$ denote a family of solutions to Problem \ref{nlProbStatementG}. Then, there exists a sequence $\{\de_j\}^{\infty}_{j = 1}$ such that $\de_j \rightarrow 0$, and a pair $(\widehat{u}, \widehat{g}) \in \Alg^{\loc}$ such that $\overline{u_{\de_j}} \rightarrow \widehat{u}$ strongly in $L^p(\Om; \R^n)$ and $\overline{g_{\de_j}} \rightharpoonup \widehat{g}$ weakly in $L^{p'}(\Om; \R^n)$, as $\de \rightarrow 0^+$.
\end{lemma}
\begin{proof}
	Since $\{\overline{g_{\de}}\}_{{\de}>0}$ belongs to $Z_{\text{ad}}$, a bounded and weakly closed subset of $L^{p'}(\Om; \R^n)$, existence of a function $\widehat{g} \in Z_{\text{ad}}$ such that $\overline{g_{\de}} \rightharpoonup \widehat{g}$ weakly in $L^{p'}(\Om; \R^n)$ immediately follows from the reflexivity of $L^{p'}(\Om; \R^n)$ and weak closedness of $Z_{\text{ad}}$.
    
	As for the states, we repeat the steps of the proof of  Lemma \ref{bddAdmissibleState} to obtain
    \begin{equation*}
\|D^{s}_{\de_j}\overline{u_{\de_j}}\|^p_{L^p(\Om; \R^{n \times n})} \ \leq \ \|\overline{g_{\de_j}}\|_{L^{p'}(\Om; \R^n)}\|\overline{u_{\de_j}}\|_{L^p(\Om; \R^n)} +\|a\|_{L^1(\Om)} +\left| W_0^{\de_j,s}(0)\right|.
       \end{equation*}    
       
       By \emph{\ref{H2}}, Theorem \ref{Gade->0: GaConv} and the Dominated Convergence Theorem we have that the term $\left| W_0^{\de_j,s}(0)\right|$ is bounded in $j$.  Using the Nonlocal Poincaré inequality Proposition \ref{Prop: NLPoincare} \emph{b)}, together with the uniform bound from the definition of $Z_{\text{ad}}$ leads to
        \begin{equation*}
        \|D^{s}_{\de_j}\overline{u_{\de_j}}\|^p_{L^p(\Om; \R^{n \times n})} \ \leq  \ C\|D^{s}_{\de_j}\overline{u_{\de_j}}\|_{L^p(\Om; \R^{n \times n})} +\|a\|_{L^1(\Om)}+\bar{C},
        \end{equation*}
    where $\bar{C},C>0$ are constants independent of $j$, from which it follows that $\{[\overline{u_{\de_j}}]_{H^{s, p, \de_j}(\Om; \R^n)}\}^{\infty}_{j = 1}$ is a bounded sequence (of real numbers). Then, due to the compactness result Lemma \ref{le:compactness}, there exists a (not relabeled) sub-sequence and a function $\widehat{u} \in W_0^{1, p}(\Om; \R^n)$, such that $\overline{u_{\de_j}} \rightarrow \widehat{u}$ strongly in $L^p(\Om; \R^n)$ as $j \rightarrow \infty$. 
 Moreover, due to the $\Ga$-convergence result Theorem \ref{Gas->1Theorem}, and observing that $\{\mathcal{W}^{\de, s}_0\}_{\de>0}$ is an equi-coercive family of functionals with respect to strong convergence in $L^p(\Om; \R^n)$, we have that minimizers of the family $\{\mathcal{W}^{\de, s}_0\}_{s < 1}$ converge to minimizers of the local energy $\mathcal{W}^{\loc}_0$. Furthermore, since $\mathcal{W}^{\de, s}_{\overline{g_{\de}}}$ and $\mathcal{W}^{\loc}_{\overline{g}}$ serve as continuous perturbations of $\mathcal{W}^{\de, s}_0$ and $\mathcal{W}^{\loc}_0$, respectively, this property extends to these energies as well; in particular, $(\widehat{u}, \widehat{g}) \in \mathcal{K}^{\loc}$.
\end{proof}

It is natural to ask whether the admissible pair obtained from Theorem\ref{De->0: compactnessSolns} is actually an optimal pair for the local problem. That is indeed the case for the strictly convex case, as next result shows for the nonlocal $p$-Laplacian problem (which falls under the scope of Theorem \ref{WPLControlThm}). The lack of uniqueness for the non-convex state equation however, complicates the comparison among the different admissible pairs, since not all of them may be obtained as limits of solutions for the nonlocal problems. 

\begin{theorem}[Convergence of minimizers as $\de \rightarrow 0^+$]\label{De->0: convMinss}   
	Let $W$ be an energy density verifying the conditions of Theorem \ref{Gade->0: GaConv} and suppose $F$ is convex in its second argument. Let $(\overline{u_{\de}}, \overline{g_{\de}})$ denote the solution of Problem \ref{nlProbStatement}, while $(\overline{u}, \overline{g})$ denotes the solution of Problem \ref{lProbStatement}. Then, as $\de \rightarrow 0^+$, we have that 
    \begin{itemize}
    \item[a)]$\overline{u_{\de}} \rightarrow \overline{u}$ strongly in $L^p(\Om; \R^n)$,
    \item[b)] $\overline{g_{\de}} \rightharpoonup \overline{g}$ weakly in $L^{p'}(\Om; \R^n)$,
    \item[c)] $D^s_\de \bar{u}_{\de} \to \nabla \bar{u}$ strongly in $L^p(\Om; \R^{n \times n})$.
\end{itemize}
Moreover, we have the limit 
	\begin{equation}\label{Eq: limde->0MinCon}
		\lim_{\de \rightarrow 0^+}\mathcal{F}(\overline{u_{\de}}, \overline{g_{\de}}) \ = \ \mathcal{F}(\overline{u}, \overline{g}).
	\end{equation}
\end{theorem}
\begin{proof}
We first prove \emph{a)} and \emph{b)}. Due to Lemma \ref{De->0: compactnessSolns}, there exists a sequence $\{\de_j\}^{\infty}_{j = 1}$ such that $\de_j \rightarrow 0$, and a pair $(\widehat{u}, \widehat{g}) \in \mathcal{K}^{\loc}$ such that $\overline{u_{\de_j}} \rightarrow \widehat{u}$ strongly in $L^p(\Om_{\de}; \R^n)$ and $\overline{g_{\de_j}} \rightharpoonup \widehat{g}$ weakly in $L^{p'}(\Om; \R^n)$, as $\de_j \rightarrow 0$. We will see that $(\widehat{u}, \widehat{g})$ is actually a minimizer of Problem \ref{lProbStatement} which, by uniqueness, coincides with $(\overline{u}, \overline{g})$. 
First we observe that every element in $\mathcal{K}^{\loc}$ can be obtained as a limit of solutions of the nonlocal problems. Indeed, given $(v, f) \in \mathcal{K}^{\loc}$ be arbitrary, and let $v_{\de}$ denote the solution to \eqref{nonlocalAdmiClassG} with right-hand side $f$, if we repeat the argument of Lemma \ref{De->0: compactnessSolns} with $\overline{g_{\de_j}} := f$ for all $\de>0$, what we obtain is a function $\widehat{v} \in W_0^{1, p}(\Om; \R^n)$ such that $v_{\de} \rightarrow \widehat{v}$ strongly in  $L^p(\Om_{\de}; \R^n)$. Since the local state equation is well-posed, we have that $\widehat{v} = v$.
	Then, by the Dominated Convergence Theorem (which applies due to the growth condition \eqref{Gqcoercive}),
	\begin{equation}\label{convMinsde->0Eq1}   
		\lim_{\de \rightarrow 0}\mathcal{F}(v_\de, f) \ = \ \mathcal{F}(v, f).
	\end{equation}
	Meanwhile, due to the minimality of $(\overline{u_\de}, \overline{g_\de})$ for each $\de>0$, we have the inequality
	\begin{equation}\label{convMinsde->0Eq2} 
		\liminf_{\de \rightarrow 0}\mathcal{F}(v_\de, f) \ \geq \ \liminf_{\de \rightarrow 0}\mathcal{F}(\overline{u_\de}, \overline{g_\de}).
	\end{equation}
	Next, we may use the weak lower semi-continuity of the $L^{p'}$ norm and the Dominated Convergence Theorem to see that
	\begin{equation}\label{convMinsde->0Eq3} 
		\liminf_{\de \rightarrow 0}\mathcal{F}(\overline{u_\de}, \overline{g_\de}) \ \geq \ \mathcal{F}(\overline{u}, \overline{g}).
	\end{equation}
	The conclusion of this inequality chain is that $\mathcal{F}(v, f) \geq \mathcal{F}(\overline{u}, \overline{g})$. Moreover, setting $v := \overline{u}$ and $f := \overline{g}$ yields the limit $\lim_{\de \rightarrow 0}\mathcal{F}(\overline{u_\de}, \overline{g_\de}) = \mathcal{F}(\overline{u}, \overline{g})$, completing the proof.

Finally, we simply comment that the proof of \emph{c} follows exactly that of part \emph{c} for Theorem \ref{convMinss->1}.

    \end{proof}

Despite the aforementioned obstacle for obtaining a similar result in the non-convex case, a relative  minimization result can be presented, with respect every admissible pair obtained as limit of solutions.
\begin{corollary}[Relative minimization]\label{de->0: compactnessRecminsQuasi}
	In the context of Theorem \ref{De->0: convMinss}. Let $\{(\overline{u_{\de}}, \overline{g_{\de}})\}_{\de > 0}$ be a family of solutions for Problem \ref{nlProbStatementG}. Then there exists a pair $(\overline{u}, \overline{g}) \in \mathcal{K}^{\loc}$ such that, up to a sub-sequence, $\overline{u_{\de}} \rightarrow \overline{u}$ strongly in $L^p(\Om; \R^n)$, and $\overline{g_{\de}} \rightharpoonup \overline{g}$ weakly in $L^{p'}(\Om; \R^n)$. Furthermore, for every $g \in Z_{\text{ad}}$, there exists a $\widetilde{u} \in W^{1, p}_0(\Om; \R^n)$ so that $(\widetilde{u}, g) \in \mathcal{K}^{\loc}$, and
	\begin{equation}\label{lem: compactnessRecminsQuasiIneqDe->0}
		\mathcal{F}(\overline{u}, \overline{g}) \ \leq \ \mathcal{F}(\widetilde{u}, g).
	\end{equation}
\end{corollary}
\begin{proof}
    Given the lack of uniqueness of the non-convex state equation, we cannot identify $\hat{v}$ with $v$ in the proof of Theorem \ref{De->0: convMinss} (just before \eqref{convMinsde->0Eq1}). Instead, combing \eqref{convMinsde->0Eq1}, \eqref{convMinsde->0Eq2} and \eqref{convMinsde->0Eq3} leads to \eqref{lem: compactnessRecminsQuasiIneqDe->0}.
\end{proof}

Finally, similarly to what happens in the localization as $s \rightarrow 1$ case, we have that the convergence of controls can be improved to strong convergence. Concretely, thanks to the structure of the cost functional, it can be improved from weak $L^{p'}(\Om; \R^n)$ convergence to strong convergence in any $L^r(\Om; \R^n)$ space (for $r \in [1, \infty)$). Furthermore, the next result actually holds by just assuming the existence of a converging sequence of minimizers to a minimizer of the local problem, regardless whether the state equation is convex or not. We omit the detailed proof since it essentially follows that of Theorem \ref{Th: StrConvControls}.

\begin{theorem}[Strong convergence of controls as $\de \rightarrow 0^+$]\label{De->0: StrConvControls}
	Let $W$ be an energy density verifying the conditions of Theorem \ref{Gas->1Theorem}. 
   Assume $\{(\overline{u_{\de}}, \overline{g_{\de}})\}_{\de > 0}$ is a sequence of solutions to Problem \ref{nlProbStatement} (respectively Problem \eqref{nlProbStatementG}), converging to $(\overline{u}, \overline{g})$, a solution of the respective local problem. Then, as $\de \rightarrow 0^+$, we have that $\overline{g_{\de}} \rightarrow \overline{g}$ strongly in $L^r(\Om; \R^n)$ for any $r \in [1, \infty]$.
\end{theorem}

\begin{remark}\label{Rmk: PathDeAndS}
    A priori, it is possible that the limits of the sequences $\{D^s_{\de}\overline{u_{ s}}\}_{s < 1}$ and $\{D^s_{\de}\overline{u_{\de}}\}_{\de > 0}$ may be different. However, in the setting of Problems \ref{nlProbStatement} and \ref{lProbStatement}, these sequences both convergence to the optimal state $\overline{u}$ for Problem \ref{lProbStatement}. Since this optimal state is unique when $F$ is non-decreasing in the second argument (see Theorem \ref{WPLControlThm}), necessarily both sequences of nonlocal gradients have sub-sequences converging to $\overline{u}$ strongly in $L^p(\Om; \R^n)$.
\end{remark}

\begin{remark}\label{Rmk: RieszGradient}
    Our study of control problems and their convergence could also be applied in the setting of Riesz fractional gradients even if $\Om$ is not bounded; the requisite compact embeddings, compactness results, and Poincaré Inequality are detailed in \cite{bellido2021gamma}. We may also generate another class of optimal control results using the kernel assumptions in \cite{schonberger2024nonlocal} and sending $\de \rightarrow \infty$ (rather than $\de \rightarrow 0^+$).  
\end{remark}


\section{Conclusion and future work}\label{sec: future}

In this paper we studied two families of nonlocal optimal control problems on function spaces parameterized by a fractional parameter and a horizon parameter. We studied the special case where the constraining energy to be minimized took the form of a nonlocal $p$-Laplacian energy, then considered which steps could and could not be replicated in the more general case where an arbitrary quasiconvex energy is minimized.

A natural extension to this project, which we may consider in the future, is the development of $\Ga$-convergence results as $\de \rightarrow 0^+$ and $s \rightarrow 1^-$ simultaneously, along with the question of whether minimizers and minimum values still converge  (for the corresponding optimal control problems). One could also consider nonlocal optimal design problems in the functional framework presented in this paper; for instance the coefficient $\mathbb{A}$ appearing in the nonlocal convex energy \eqref{Eq: pLapDens} could serve as the control variable, in a similar spirit to the paper \cite{mengesha2025design}.

In addition, one of the main drawbacks to using $\Ga$-convergence is that it does not admit rates of convergence (in $\de$ or $s$). One must instead utilize regularity theory to assure that solutions to the problems belong to higher-order fractional spaces, and then rates of convergence can be recovered (see the recent work \cite{scott2025nonlocal} that considers this problem for a different class of nonlocal spaces). This is also a worthwhile extension of the present work.

Finally, a problem of mathematical interest is to consider control problems where the cost functional  does not adhere to any notion of convexity. This would require the use of second-order optimality conditions akin to the papers \cite{antil2018brief, antil2020optimal}.


\section*{Acknowledgments}

\noindent The work of JC has been supported by Fundaci\'on Ram\'on Areces. JC has also been supported by the Spanish Agencia Estatal de Investigaci\'on through the project  PID2024-158664NB-C2. Meanwhile, JMS did not receive any external funding that was used to complete this project.

\bibliographystyle{plain} 
\bibliography{refs} 

\end{document}